\documentclass{article}

\usepackage{lmodern}
\usepackage[utf8]{inputenc}
\usepackage[T1]{fontenc}
\usepackage[english]{babel}
\usepackage[normalem]{ulem}

\usepackage[colorlinks=true]{hyperref}
\usepackage{pdfsync}
\usepackage[usenames, dvipsnames]{xcolor}
\usepackage{tikz}
\usepackage{pgfplots}
\usepackage{subfig}
\usepackage{stmaryrd}
\usepackage{microtype}
\usepackage{multirow}
\usepackage{float}
\usepackage{bbm}
\usepackage{color}

\usepackage{amsmath,amssymb,amsthm,graphicx,epsfig,float,url}
\usepackage{mathtools,amssymb,amsfonts,mathrsfs}

\usepackage{stackrel}
\usepackage{bbm} 
\mathtoolsset{showonlyrefs}
\usepackage{pgfplots}
\usepackage{graphicx}
\graphicspath{{Figures/}}
\usepackage{caption}
\usepackage{microtype}

\usepackage{booktabs}
\usepackage{multirow}
\usepackage{rotating,tabularx}
\usepackage{adjustbox}
\usepackage{enumitem}
\usepackage{bigints}
\usepackage{threeparttable, tablefootnote}
\usepackage{comment}
\usepackage{wasysym}

\newcommand{\charac}[1]{\mathbbm{1}_{#1}}
\newcommand{\lrp}[1]{\left( #1 \right)}
\newcommand{\R}{\textnormal{I\kern-0.21emR}}
\newcommand{\N}{\textnormal{I\kern-0.21emN}}

\renewcommand{\leq}{\leqslant}

\newtheorem{theorem}{Theorem}  
\newtheorem{proposition}{Proposition}

\theoremstyle{definition}\newtheorem{remark}{Remark}

\newcommand{\enum}[2]{\left\{#1,\dots,#2\right\}}

\title{\bf Vector-borne disease outbreak control via instant releases}

\author{Luis Almeida\footnote{Sorbonne Universit\'e, CNRS, Universit\'e de Paris, Inria, Laboratoire J.-L. Lions, 75005 Paris, France ({\tt luis.almeida@sorbonne-universite.fr}).}
	\and Jes\'us Bellver Arnau\footnote{Sorbonne Universit\'e, CNRS, Universit\'e de Paris, Inria, Laboratoire J.-L. Lions, 75005 Paris, France ({Current adress: Centre d’Estudis Avan\c{c}ats de Blanes (CEAB-CSIC), Carrer d’Accés a la cala Sant Francesc 14, 17300 Blanes, Spain, \tt jesus.bellver@ceab.csic.es}).}
	\and Yannick Privat\footnote{IRMA, Universit\'e de Strasbourg, CNRS UMR 7501, Inria, 7 rue Ren\'e Descartes, 67084 Strasbourg, France ({\tt yannick.privat@unistra.fr}).}\textsuperscript{~~}\footnote{Institut Universitaire de France (IUF).} 
	\and Carlota Rebelo\footnote{Departamento de Matem\'atica and CEMAT-Ci\^encias, Faculdade de Ci\^encias, Universidade de Lisboa, Campo Grande, 1749-016 Lisbon, Portugal ({\tt mcgoncalves@fc.ul.pt}).}
}

\date{}

\begin{document}
\maketitle

	\begin{abstract}
	This paper is devoted to the study of optimal release strategies to control vector-borne diseases, such as dengue, Zika, chikungunya and malaria. Two techniques are considered: the sterile insect one (SIT), which consists in releasing sterilized males among wild vectors in order to perturb their reproduction, and the {\it Wolbachia} one (presently used mainly for mosquitoes), which consists in releasing vectors, that are infected with a bacterium limiting their vector capacity, in order to replace the wild population by one with reduced vector capacity. In each case, the time dynamics of the vector population is modeled by a system of ordinary differential equations in which the releases are represented by linear combinations of Dirac measures with positive coefficients determining their intensity. We introduce optimal control problems that we solve numerically using ad-hoc algorithms, based on writing first-order optimality conditions characterizing the best combination of Dirac measures. We then discuss the results obtained, focusing in particular on the complexity and efficiency of optimal controls and comparing the strategies obtained. Mathematical modeling can help testing a great number of scenarios that are potentially interesting in future interventions (even those that are orthogonal to the present strategies) but that would be hard, costly or even impossible to test in the field in present conditions.
	\end{abstract}
	

\section{Introduction} 

Vector-borne diseases have a large impact on human health around the world, representing 17\% of all infectious diseases. These diseases can be due to parasites, bacteria or viruses and be transmitted by different types of vectors like, for instance, ticks, fleas or mosquitoes. A significant part of the models presented in this paper are applicable in a general setting. In particular, the part concerning the Sterile Insect Technique (SIT) is applicable to any vector borne disease where male vectors do not transmit the disease and where the vector has sexual reproduction which will be significantly perturbed by the release of sterile males. 

Many of these diseases, such as dengue, Zika, chikungunya, yellow fever or the West Nile fever are caused by arboviruses. The vector responsible for the transmission of many arboviruses are the mosquitoes of the genus \textit{Aedes}, specially the species \textit{Aedes Aegipty} and \textit{Aedes Albopictus}. Dengue is the most prevalent of these diseases, with more than 3.9 billion people in over 129 countries at risk of contracting it, and an estimated 40,000 death toll every year according to the World Health Organization \cite{who}. Since, at present, there is no effective vaccine or antiviral drug, the only treatment option is to relieve the symptoms. As for preventing the spread of the disease, current methods consist of directly targeting the vector.

In the fight against arboviruses, and in particular dengue, two of the main control techniques targeting the mosquitoes are the SIT and the use of {\it Wolbachia}. Both methods rely on introducing mosquitoes into the wild population with certain modifications, which allow to control the infections. The SIT consists on the release of large amounts of sterile male mosquitoes in order to reduce the mosquito population by mating with the females in the place of the fertile ones. This technique has been both studied mathematically (see, for instance, \cite{bliman_sit})  and used in the field (for the mosquito case see, for instance, \cite{sterile_releases_bellini, sterile_release_harris}), not only with mosquitoes but also with other pests (like the success in eliminating of screw worms from Curacao and then from North-America, or to fight the tsetse fly in Africa). Therefore, our approach for the SIT case can be extended to yield fruitful results for controlling other vector populations.

The \textit{Wolbachia} technique has mostly been used for \textit{Aedes} mosquitoes (and this is the context in which we chose to present it in this paper) but there are also many promising signs indicating that it should be possible to use it for other types of mosquitoes or even other vectors \cite{Wolbachia_Nature}. The release of Wolbachia-infected mosquitoes on the other hand, does not seek to eradicate the mosquito population, but rather to replace it with a new one, less capable of transmitting several diseases. Thus, both males and females need to be released in order to establish a new population.

The technique is based in the use of the bacterium {\it Wolbachia}. This bacterium, naturally present in many species of Arthropodes, although not in \textit{Aedes} mosquitoes, has been proven effective in reducing the vector capacity of mosquitoes for several arboviruses \cite{walker,wolbachia_arboviruses}. Its use as a tool for epidemic control relies on the fact that {\it Wolbachia} is vertically transmitted from the mother to the offspring, which makes the new population self-sustainable. This method also takes advantage of a phenomenon called cytoplasmic incompatibility \cite{walker}. This phenomenon produces a crossed sterility between infected males and uninfected females, which can lead to a fast spread of the bacteria in the mosquito population despite the fact that {\it Wolbachia} may also shorten the lifespan and reduce the fertility of its hosts. {\it Wolbachia} releases have been shown to be successful both in establishing the new population \cite{Deployment_Townsville,Deployment_Yogyakarta_Stable} and in reducing the number of dengue cases \cite{Deployment_Cairns,Deployment_Yogyakarta_Dengue}.

Our main goal in this work is to take advantage of the wide possibilities offered by mathematical modeling, to study and compare the effect of these techniques in interaction with the disease dynamics and consider a large variety of scenarios that can give rise to new strategies to mitigate the effects of vector-borne disease outbreaks using mosquito releases. Presently, these techniques are used mainly in a preventive setting to avoid to be in conditions where a future epidemic could propagate and are not usually applied in an epidemiological outbreak context since they may take a few mosquito generations to have a significant impact on the mosquito population which can be a problem if the disease is already present and spreading fast. In case of an outbreak, other alternatives with more immediate effects exist, like the use of pesticides but these lead to a significant negative environmental impact. Moreover, they lead to a rapid development of resistance to these products in the mosquito population which reduces to a major reduction of the effect of insecticides in the mid and long term. 

Since the population replacement technique using \textit{Wolbachia} requires the release of female mosquitoes, which, though being much poorer vectors than its wild counterparts, might still transmit it in certain cases, it raises ethical questions when used in the case a virus is actively circulating in a population. 

For all these reasons, this work should be seen as a first step towards a better understanding of the effects of modified mosquito releases in epidemiological contexts, opening the debate around broadening the scope of application of these techniques.

Since the releases occur in a much shorter time scale than the duration of the outbreak, they will be considered instantaneous. Therefore, impulsive controls are a natural setting to model field releases of this type. Note that there is work to provide a well-posed framework, existence and regularity results to optimal control problems involving impulsive controls. This work has been developed in several articles such as \cite{Miller1,Miller2,Motta_Rampazzo1,Motta_Rampazzo2,Wolenski}. We provide a detailed proof of formulas for determining various quantities of interest, so that this work is self-contained, in Sections \ref{sec:control} and \ref{sec:optcond}. As stated before, our models are valid in a much wider setting but, for the sake of clarity, for the remaining of the paper we will describe them in the setting of arboviruses and of \textit{Aedes} mosquitoes as vectors. Although with several differences, previous works model and study the arboviruses transmission between {\it Wolbachia}-infected mosquitoes, wild-type mosquitoes and humans \cite{ndii_hickson,hughes_briton}. A previous study of optimal control related issues, considering only bang-bang controls, can be found in \cite{zhang_lui}.

In order to model the virus dynamics between mosquitoes and humans we consider a SEIR (Susceptible-Exposed-Infectious-Recovered) model for the humans and a SEI model for the mosquitoes (their short lifespan leads us to neglect the recovered compartment for the mosquitoes). Concerning the population dynamics we assume the humans to have the same birth and death rate and consider a logistic growth with a death term for the mosquitoes. The human and mosquito populations are identified by using subscripts $H$ and $M$, respectively.

\begin{equation}
\label{sys:SEIR}
\begin{array}{lll}
S'_H &= & b_H H-\displaystyle\frac{\beta_{MH}}{H} I_M S_H-b_H S_H\\
E'_H &= & \displaystyle\frac{\beta_{MH}}{H} I_M S_H-\gamma_H E_H-b_H E_H \\
I'_H & = & \gamma_H E_H-\sigma_H I_H-b_H I_H\\
S'_M & = & b_M M\left(1-\displaystyle\frac{M}{K}\right)-\displaystyle\frac{\beta_{HM}}{H}S_M I_H-d_MS_M\\
E'_M & = &\displaystyle\frac{\beta_{HM}}{H} S_M I_H-\gamma_M E_M-d_ME_M\\
I'_M & = &\gamma_M E_M-d_MI_M\\
\end{array}
\end{equation}

The positive parameters used in system \eqref{sys:SEIR} are:
\begin{itemize}
	\item $b_H$, $b_M$, the birth rates for humans and mosquitoes.
	\item $d_M$, the death rate for mosquitoes. For humans the death rate is assumed to be equal to the birth rate.
	\item $\beta_{MH}$ and $\beta_{HM}$ are, respectively, the rate of mosquito bites giving rise to a transmission between infected mosquitoes and humans and between infected humans and mosquitoes.
	\item $\gamma_H$ and $\gamma_M$ are the progression rates from latent to infectious compartments in humans and mosquitoes, respectively.
	\item $\sigma_H$ is the recovery rate from the disease (for humans).
	\item $K$ is the carrying capacity of the mosquito population. It is related to the maximal amount of mosquitoes that the environment can sustain.
	\item $H$ is the total amount of humans, $H=S_H+E_H+I_H+R_H$. 
	\item $M$ is the total amount of mosquitoes, $M=S_M+E_M+I_M$.
\end{itemize}
The equation for the recovered human reads $R'_H=\sigma_H I_H-b_H R_H$. Since $H$ is constant we can remove $R_H$ from the system of differential equations and compute it as $R_H=H-S_H-E_H-I_H$. {System \eqref{sys:SEIR} can be used for modeling, a priori, any vector-borne disease without other means of transmission and for which reinfection cannot occur.}

In order to study these disease controlling techniques we need to modify this basic system in a way that takes into account the particularities of each one of them.

\begin{remark}\label{rem:males}
	It is important to remark that throughout the paper whenever we refer to `mosquitoes' we are referring exclusively to the \textit{female} mosquitoes, unless the contrary is specified. Male mosquitoes do not bite humans and therefore do not transmit the diseases considered here. Thus, the variables referring to the mosquitoes such as $S_M$, $I_M$ or $I_M$, refer to female mosquitoes. An exception being when the SIT is treated (see section \ref{sec:SIT_Introduction}). In the SIT only male mosquitoes are released, thus, $M_S$ will refer to \textit{male} mosquitoes. In order to be able to do this simplification, we assume that male and female population have the same dynamics. We assume that the probability at birth of female and male is the same (50\%) and that they both have the same life expectancy ($d_{\male}=d_{\female}=d_M$).  These assumptions are not necessarily true in field conditions, but they simplify considerably the study and are thus reasonable to do in this work where our point is to illustrate how mathematical modeling can be used as a tool to conceive alternative strategies. We prefer to work in this "toy model" setting to avoid extra complexity that would complicate the study and even risk blurring some basic qualitative behaviors. Of course, if such studies were to be conducted for conceiving a specific field intervention, it would be necessary to consider more elaborate and realistic models to have more precise estimates.
\end{remark}

\subsection{The sterile insect technique}\label{sec:SIT_Introduction}
To model the effects of the addition of sterile mosquitoes to the system we have to add an equation for them and a term accounting for the interaction between them and the mosquito population. Following the same approach as in \cite{JDE} we introduce the following system

\begin{equation}
\label{sys:Sterile}
\tag{$SIT$}
\begin{array}{lll}
S'_H &= & b_H H-\displaystyle\frac{\beta_{MH}}{H} I_M S_H-b_H S_H\\
E'_H &= & \displaystyle\frac{\beta_{MH}}{H} I_M S_H-\gamma_H E_H-b_H E_H \\
I'_H & = & \gamma_H E_H-\sigma_H I_H-b_H I_H\\
S'_M & = & b_M M\left(1-\displaystyle\frac{M}{K}\right)\displaystyle\frac{M}{M+s_c M_S}-\displaystyle\frac{\beta_{HM}}{H}  S_M I_H-d_MS_M\\
E'_M & = &\displaystyle\frac{\beta_{HM}}{H}  S_M I_H-\gamma_M E_M-d_ME_M\\
I'_M & = &\gamma_M E_M-d_MI_M\\
M'_S & = &u-d_S M_S 
\end{array}
\end{equation}

Since sterile mosquitoes don't reproduce we only consider a death term and the function $u$, representing the rate at which sterile mosquitoes are introduced in the population and interpreted as a control term for this system. We also add a birth term in the susceptible mosquitoes compartment, proportional to the probability that a female mosquito encounters a fertile male to mate (assuming that there are the same amount of male and female mosquitoes in the wild population). The positive parameter $s_c$ accounts for the competitiveness of the sterile mosquitoes since female mosquitoes may be less inclined to mate with them. This parameter presents a huge variation in the literature, from works estimating it to be low ($s_c=0.14$ in \cite{Competitiveness_yes}) to works where no difference in competitiveness was found \cite{Competitivness_no}. According to \cite{Competitiveness_yes}, it would be relevant to assume the parameter $s_c$ depending on the ratio of sterile to fertile mosquitoes which would imply $s_c=s_c(M_S/M)$. Nevertheless, for simplicity, here we will assume it to be constant. 

Although vector-borne diseases are transmitted, both from human to mosquito and from mosquito to human, through the vector's bite, the biological process by which they are transmitted in each case is rather different. Mosquito acquires the virus through the blood of the host, while humans get infected through the saliva of the infected mosquito. Despite this fact, we have not found conclusive evidence to expect $\beta_{MH}$ to be higher or smaller than $\beta_{HM}$, therefore, for the numerical simulations of Section \ref{sec:results}, we use the same numerical value. Nonetheless, to keep the theoretical analysis compatible with the possibility that they may have different values, we keep the distinction between both quantities in the following computations. Note that there is no need to consider the dynamics of dengue in the sterile mosquito population, since the released mosquitoes are only male and, therefore, they do not feed on human blood. Thus, they are not vectors for disease transmission between humans.

\subsection{The {\it Wolbachia} method}
In this case we add a second mosquito population. This new population is composed of mosquitoes carrying Wolbachia, and the related quantities will be subscripted by $W$. It has been shown that {\it Wolbachia} decreases the fecundity and increases the mortality rates of mosquitoes \cite{walker}, therefore $b_W<b_M$ and $d_W>d_M$. Also, {\it Wolbachia} reduces the vector capacity of the mosquitoes. We thus introduce $0<\beta_{WH}<\beta_{HW}<\beta_{HM}=\beta_{MH}$ to make the distinction between the rate of mosquito bites giving rise to a transmission from human to Wolbachia-carrying mosquitoes, $\beta_{HW}$, and the rate of mosquito bites giving rise to a transmission from Wolbachia-carrying mosquitoes to humans, $\beta_{WH}$. Contrary to the previous case, we do make the distinction between both quantities also numerically in this case. The first one, $\beta_{HW}$, is smaller than $\beta_{MH}$, $\beta_{HM}$ since {\it Wolbachia} affects the capability of mosquitoes to feed due to a deformation in the trunk \cite{trunk}. The second one, $\beta_{WH}$, should be smaller than the first one since {\it Wolbachia} also affects the way the disease develops inside the body of the mosquitoes and reduces the viral load in their saliva  \cite{betaWH_moreira_iturbe,betaWH_bian_xu}. We also introduce the term $1-s_h\frac{W}{M+W}$ to take into account the cytoplasmic incompatibility. Here, $s_h$ represents the level of cytoplasmic incompatibility achieved by the strain of \textit{Wolbachia}. We have $0\leq s_h\leqslant 1$, with $s_h=0$ meaning that there is not any incompatibility and $s_h=1$ meaning that the incompatibility is perfect. Finally, we introduce $\gamma_W$ since {\it Wolbachia} also delays the amount of time it takes for dengue virus to reach the saliva of the mosquitoes, thus lengthening the effective incubation period of the disease in the mosquitoes carrying it \cite{EIP}.

\begin{equation}
\label{sys:Wolbachia}
\tag{$WB$}
\begin{array}{lll}
S'_H &= & bH-\displaystyle\frac{\beta_{MH}}{H} I_M S_H-\displaystyle\frac{\beta_{WH}}{H} I_W S_H-b_HS_H\\
E'_H &= & \displaystyle\frac{\beta_{MH}}{H} I_M S_H+\displaystyle\frac{\beta_{WH}}{H} I_W S_H -\gamma_H E_H-b_H E_H \\
I'_H & = & \gamma_H E_H-\sigma_H I_H-b_H I_H\\
S'_M & = & b_MM\left(1-\displaystyle\frac{M+W}{K}\right)\left(1-s_h\displaystyle\frac{W}{M+W}\right)-\displaystyle\frac{\beta_{HM}}{H} S_M I_H-d_MS_M\\
E'_M & = &\displaystyle\frac{\beta_{HM}}{H} S_M I_H-\gamma_M E_M-d_ME_M\\
I'_M & = &\gamma_M E_M-d_MI_M\\
S'_W & = & b_W W\left(1-\displaystyle\frac{M+W}{K}\right)-\displaystyle\frac{\beta_{HW}}{H} S_W I_H-d_WS_W+u\\
E'_W & = &\displaystyle\frac{\beta_{HW}}{H} S_W I_H-\gamma_W E_W-d_WE_W\\
I'_W & = &\gamma_W E_W-d_WI_W
\end{array}
\end{equation}

Before moving on to the control problem we perform two simplifications on the system. We consider the following variables: $M:=S_M+E_M+I_M$ and $W:=S_W+E_W+I_W$. These variables account for the mosquito population regardless of the dengue dynamics. These variables present the following dynamics
\begin{equation}
\label{eq:pop_dyn}
\begin{array}{lll}
M' & = & b_MM\left(1-s_h\displaystyle\frac{W}{M+W}\right)\left(1-\displaystyle\frac{M+W}{K}\right)-d_M M\\
W' & = & b_WW\left(1-\displaystyle\frac{M+W}{K}\right)-d_W W+u\\
\end{array} 
\end{equation}
These equations describing the population dynamics of the mosquitoes in our model are those of the model in \cite{wolbachia}. One can observe looking at the values in table \ref{tab:parameters} that $b_M\gg d_M$ and $b_W\gg d_W$. That is, that the birth rate of the mosquitoes is much higher than the death rate in both populations.
In \cite[Prop.~1]{wolbachia}, it is proven that in the high birth rate limit, i.e. considering $b_M=b_M^0/\varepsilon$, $b_W=b_W^0/\varepsilon$ and taking the limit $\varepsilon\to 0$, the proportion of mosquitoes $p=W/(M+W)$ converges uniformly to the solution of a simple equation on the proportion of {\it Wolbachia}-infected mosquitoes.
Hence, the asymptotic system \eqref{eq:pop_dyn} reads:
\begin{equation*}
p'=f(p)+ug(p).
\end{equation*}
where $$f(p)=p(1-p)\frac{d_M b_W^0-d_Wb_M^0(1-s_hp)}{b_M^0(1-p)(1-s_hp)+b_W^0p}$$ and $$g(p)=\frac{1}{K}\frac{b_M^0(1-p)(1-s_hp)}{b_M^0(1-p)(1-s_hp)+b_W^0p}.$$ 
Another consequence is that $M+W$ converges to $K$ and so, in the limit, $W=(M+W)\frac{W}{M+W}=Kp$, and therefore $M=K(1-p)$.

This limit leaves the equations for the humans and for the infected mosquitoes unchanged. In order to modify the equations for the latent mosquitoes we can straightforwardly set $M+W=K$. Finally, using that $S_M=M-E_M-I_M$ and $S_W=W-E_W-I_W$ we can eliminate the two equations for the susceptible mosquitoes from the system. The equations for the exposed mosquitoes become:

\begin{equation}
\label{eq:exposed}
\begin{array}{lll}
E'_M & = & \displaystyle\frac{\beta_{HM}}{H}  (K(1-p)-E_M-I_M) I_H-\gamma_M E_M-d_ME_M\\
E'_W & = & \displaystyle\frac{\beta_{HW}}{H} (Kp-E_W-I_W) I_H-\gamma_W E_W-d_WE_W
\end{array}
\end{equation}

Incorporating these changes into system \eqref{sys:Wolbachia} we obtain the system we are going to study

\begin{equation}
\label{sys:Wolbachia_Simp}
\tag{$WB'$}
\begin{array}{lll}
S'_H &= & bH-\displaystyle\frac{\beta_{MH}}{H} I_M S_H-\displaystyle\frac{\beta_{WH}}{H} I_W S_H-b_HS_H\\
E'_H &= & \displaystyle\frac{\beta_{MH}}{H} I_M S_H+\displaystyle\frac{\beta_{WH}}{H} I_W S_H -\gamma_H E_H-b_H E_H \\
I'_H & = & \gamma_H E_H-\sigma_H I_H-b_H I_H\\
E'_M & = &\displaystyle\frac{\beta_{HM}}{H}  (K(1-p)-E_M-I_M) I_H-\gamma_M E_M-d_ME_M\\
I'_M & = &\gamma_M E_M-d_MI_M\\
E'_W & = &\displaystyle\frac{\beta_{HW}}{H} (Kp-E_W-I_W) I_H-\gamma_W E_W-d_WE_W\\
I'_W & = &\gamma_W E_W-d_WI_W\\
p' &=& f(p)+ug(p)
\end{array}
\end{equation}

\section{Study of the uncontrolled system}\label{sec:uncontrolled}
In this section we study the uncontrolled systems (setting $u=0$ for all $t\in[0,T]$) and compute the equilibria and the per stage reproduction number (given by the next generation technique), $R_0$, of dengue in each case. This $R_0$ is a useful tool in the study of epidemiological systems with two stages, in this case host-vector and vector-host. It stands for  the number of secondary infections generated per stage in a population where all individuals are susceptible to the disease ($S_H=H$ and $S_M=$ total population of vectors), which is the setting in which we will perform the numerical simulations. This number is the square root of the basic reproduction number \cite[Page 110]{martcheva}.

\subsection{Sterile insect technique}
Since we consider $u=0$ and $M_S(0)=0$, it follows that $M_S(t)=0$ for all $t\in[0,T]$, turning system \eqref{sys:Sterile} into \eqref{sys:SEIR}. So computing the equilibria and $R_0$ of this system boils down to computing those of system \eqref{sys:SEIR}. 
We compute now the equilibria of system \eqref{sys:SEIR}. Since $M=S_M+E_M+I_M$, for any equilibrium $(S_M^*,E_M^*,I_M^*)$ of the system, we have that $M^*=S_M^*+E_M^*+I_M^*$ must also be an equilibrium of the equation
\begin{equation}
M'  =  b_M M\left(1-\displaystyle\frac{M}{K}\right)-d_M M .
\end{equation}
This equation presents two equilibria, $M^*=0$ and $M^*=K(1-d_M/b_M)$. We can use this to simplify the study of the equilibria of system \eqref{sys:SEIR}. The system to solve becomes

\begin{equation}
\begin{array}{lll}
0 &= & b_H H-\displaystyle\frac{\beta_{MH}}{H} I_M^* S_H^*-b_H S_H^*\\
0 &= & \displaystyle\frac{\beta_{MH}}{H} I_M^* S_H^* -\gamma_H E_H^*-b_H E_H^* \\
0 & = & \gamma_H E_H^*-\sigma_H I_H^*-b_H I_H^*\\
0 & = & d_M M^*-\displaystyle\frac{\beta_{HM}}{H} S_M^* I_H^*- d_M S_M^*\\
0 & = &\displaystyle\frac{\beta_{HM}}{H} S_M^* I_H^*-\gamma_M E_M^*-d_M E_M^*\\
0 & = &\gamma_M E_M^*-d_M I_M^*\\
\end{array}
\end{equation}

Solving this simpler system we obtain three different equilibria:
\begin{itemize}
	\item The extinction (vector-free) equilibrium $$(S_H^*,E_H^*,I_H^*,S_M^*,E_M^*,I_M^*)=\lrp{H,0,0,0,0,0}$$
	\item The disease-free equilibrium $$(S_H^*,E_H^*,I_H^*,S_M^*,E_M^*,I_M^*)=\lrp{H,0,0,K^*,0,0}$$
	\item The endemic equilibrium 
		$$(S_H^*,E_H^*,I_H^*,S_M^*,E_M^*,I_M^*)=\lrp{H-a_H I_H^*,\frac{\sigma_H+b_H}{\gamma_H}I_H^*,I_H^*,K^*-a_MI_M^*,\frac{d_M}{\gamma_M}I_M^*,I_M^*},$$
	where $a_H=\frac{(\gamma_H+b_H)(\sigma_H+b_H)}{b_H\gamma_H}$, $a_M=\frac{\gamma_M+d_M}{\gamma_M}$, $$ I_H^*=\frac{K^* \beta_{MH}}{H b_H a_M + K^*\beta_{MH}}\lrp{1-\frac{1}{(R_0^M)^2}}\frac{H}{a_H},$$ 
	and $$I_M^*=\frac{\beta_{HM}}{a_H d_M+\beta_{HM}}\lrp{1-\frac{1}{(R_0^M)^2}}\frac{K^*}{a_M}.$$
	It is enligthening to write the endemic equilibrium in terms of the $R_0$ of the system (that we denote $R_0^M$), since it clearly shows that if $R^M_0<1$ the endemic equilibrium does not exist.
\end{itemize}
	
	In order to compute $R_0^M$, we proceed as in \cite{R0_comp}. The relevant compartments for these computations are only the infected ones. As \cite{R0_comp} points out, this distinction is determined from the epidemiological interpretation of the model and cannot be deduced from the structure of the equations alone. The infected compartments in our case are those in which there are individuals carrying the dengue virus. For model \eqref{sys:SEIR} these  are $E_H$,$I_H$,$E_M$ and $I_M$.
	
	We need then to separate the changes in the compartments due to new infections from the rest. We write system \eqref{sys:SEIR} in the following way $\mathbf{x} =(E_H,I_H,E_M, \allowbreak I_M, S_H, S_M)$, $$\dot{\mathbf{x}}=\mathcal{F}(\mathbf{x})-\mathcal{V}(\mathbf{x}),$$ where  $\mathcal{F}$ contains the rate of appearance of new infections in each compartment and $\mathcal{V}$ the rate of transfer of individuals into the compartments by all other means. 
	
	Let's see the decomposition of the first equation, $E_H'$, as an example:
	$$E_H'= \underbrace{\displaystyle\frac{\beta_{MH}}{H} I_M S_H}_{\mathcal{F}_1(\mathbf{x})}-\underbrace{(\gamma_H E_H+b_H E_H)}_{\mathcal{V}_1(\mathbf{x})}.$$
	Doing this decomposition for all the equations we obtain
	$$\mathcal{F}(\mathbf{x})=\left(\displaystyle\frac{\beta_{MH}}{H} I_M S_H,0,\displaystyle\frac{\beta_{HM}}{H} S_M I_H,0,0,0\right)$$
	and $\mathcal{V}(\mathbf{x})$ containing all the other terms.
	
	Then we construct the matrices $$F=\left(\frac{\partial\mathcal{F}_i}{\partial x_j}(\mathbf{x_0})\right)_{i,j} \quad \mbox{and} \quad V=\left(\frac{\partial\mathcal{V}_i}{\partial x_j}(\mathbf{x_0})\right)_{i,j} \quad \mbox{,} \quad i,j=1,\dots,4 $$ where $\mathbf{x_0}$ represents the equilibrium for which we compute the $R_0$, i.e., the disease-free equilibrium $\mathbf{x_0}=(0,0,0,0,H,K^*)$. The values taken by $i$ and $j$ are given by the fact we labeled the infected compartments $1$ to $4$. These matrices for our model read
	
	$$F=\begin{pmatrix}
		0 & 0 & 0 & \beta_{MH} \\
		0 & 0 & 0       & 0 \\
		0 & \displaystyle\frac{\beta_{HM}}{H}K^* & 0 & 0 \\
		0 & 0 & 0       & 0 \\
	\end{pmatrix} \quad \mbox{and} \quad 
	V=\begin{pmatrix}
		\gamma_H+b_H & 0 & 0 & 0 \\
		-\gamma_H & \sigma_H+b_H & 0 & 0 \\
		0 & 0 & \gamma_M+d_M & 0 \\
		0 & 0 & -\gamma_M & d_M \\
	\end{pmatrix}.$$
	
	In \cite{R0_comp} is shown that $R_0=\rho(F V^{-1})$, where $\rho$ denotes the spectral radius of the resulting matrix, namely 
	\begin{equation}\label{def:R0M}
		R^M_0:=\rho(F V^{-1})=\sqrt{\frac{\beta_{HM}\beta_{MH} K^*\gamma_M \gamma_H}{H d_M(b_H+\sigma_H)(\gamma_M+d_M)(\gamma_H+b_H)}}.
	\end{equation}
	where $K^*=K(1-d_M/b_M)$. For the parameters considered in Table \ref{tab:parameters} we find $R_0^M\approx 1.67$, {which gives a basic reproduction number of  $\lrp{R_0^M}^2\approx 2.80$.}

\subsection{{\it Wolbachia} method}\label{subsec:wolb_uncontrolled}

Since the equation $p'=f(p)$ is independent of the rest we can solve it separately. The function $f(p)$ has only three zeros, $p^*=0$, $p^*=1$ and $p^*=\theta$, satisfying $0<\theta<1$. The last zero only exists if we further assume that $1-s_h<\frac{d_Mb_W^0}{d_Wb_M^0}<1$, which is satisfied in our case. The value of $\theta$ can be computed from the parameters of the problem, yielding $\theta=\frac{1}{s_h}\lrp{1-\frac{d_M b_W^0}{d_W b_M^0}}$. This implies that, independently of the epidemiological part of the model, there exists a {\it Wolbachia}-free equilibrium, a full invasion equilibrium and a coexistence equilibrium in the mosquito population.

We compute now the solutions of 
\begin{equation}
	\label{sys:6}
	\begin{array}{lll}
		0 &= & b_H H-\displaystyle\frac{\beta_{MH}}{H} I_M^*S_H^*-\displaystyle\frac{\beta_{WH}}{H} I_W^* S_H^*-b_H S_H^*\\
		0 &= &\displaystyle\frac{\beta_{MH}}{H} I_M^*S_H^*+\displaystyle\frac{\beta_{WH}}{H} I_W^* S_H^* -\gamma_H E_H^*-b_H E_H^* \\
		0 & = & \gamma_H E_H^*-\sigma_H I_H^*-b_H I_H^*\\
		0 & = &\displaystyle\frac{\beta_{HM}}{H}(K(1-p^*)-E_M^*-I_M^*) I_H^*-\gamma_M E_M^*-d_ME_M^*\\
		0 & = &\gamma_M E_M^*-d_MI_M^*\\
		0 & = &\displaystyle\frac{\beta_{HW}}{H} (Kp^*-E_W^*-I_W^*) I_H^*-\gamma_W E_W^*-d_WE_W^*\\
		0 & = &\gamma_W E_W^*-d_WI_W^*\\
	\end{array}
\end{equation}
as functions of $p^*$.
Let us define  $a_W:=\frac{\gamma_W+d_W}{\gamma_W}$,
\begin{equation}\label{def:R0W}
	R_0^W:=\sqrt{\frac{\beta_{HW}\beta_{WH}K\gamma_W \gamma_H}{Hd_W(b_H+\sigma_H)(\gamma_W+d_W)(\gamma_H+b_H)}},
\end{equation}
and $R_0^M$ as defined in \eqref{def:R0M} but using $K$ instead of $K^*$. Note that in the high birth rate limit, $K^*=K\lrp{1-\frac{d_M}{b_M}}$ tends to $K$. These $R_0^M$ and $R_0^W$ are the per stage reproduction numbers associated with the disease-free equilibria, for $p^*=0$ (Wolbachia-free) and $p^*=1$ (full invasion) respectively. Let us also define $R_{p^*}^2:=\lrp{R_0^W}^2p^*+\lrp{R_0^M}^2(1-p^*)$ {(an analogous closed formula was considered in \cite{cardona_salgado}).}
We find that system \eqref{sys:6} has the trivial solution $(S_H^*,E_H^*,I_H^*,E_M^*,I_M^*,E_W^*,I_W^*)=(H,0,0,0,0,0,0)$, which gives three different equilibria for system \eqref{sys:Wolbachia_Simp}: $(H,0,0,0,0,0,0,0)$, $(H,0,0,0,0,0,0,\theta)$ and $(H,0,0,0,0,0,0,1)$.

In case $R_{p^*}>1$, system \eqref{sys:6} presents another real solution
\[
\lrp{H-a_H I_H^*,\frac{\sigma_H+b_H}{\gamma_H}I_H^*,Hr,\frac{d_M}{\gamma_M}I_M^*,\frac{K}{a_M}\frac{\beta_{HM} r}{d_M + \beta_{HM} r}(1-p^*),\frac{d_W}{\gamma_W}I_W^*,\frac{K}{a_W}\frac{\beta_{HW} r}{d_W + \beta_{HW} r}p^*},
\]
where $r$ is the positive root of the second order polynomial
\begin{eqnarray*}
	P(Z)&=&Z^2\lrp{\beta_{HM}\beta_{HW}+a_H\lrp{\beta_{HM} d_W \lrp{R_0^M}^2(1-p^*)+\beta_{HW} d_M \lrp{R_0^W}^2 p^*}}\\
	&&+Z\lrp{R_{p^*}^2 a_H d_M d_W-\lrp{R_{p^*}^2-1}\lrp{\beta_{HM} d_W +\beta_{HW} d_M}}- \lrp{R_{p^*}^2-1}d_M d_W.
\end{eqnarray*} 
That means that system \eqref{sys:Wolbachia_Simp} can have up to six equilibria, due to the fact that there are three different values of $p^*$ and that $R_{p^*}$ can be bigger than one for some values of $p^*$ but not for others. 

Since in system \eqref{sys:Wolbachia_Simp} mosquitoes with {\it Wolbachia} are also present, there are six infected compartments and two relevant $R_0$, one at the disease-free/{\it Wolbachia}-free equilibrium, $R_0^M$, and one at the disease-free/full invasion equilibrium, $R_0^W$. We follow, step by step, the same procedure, adapting it to the new system for each of the $R_0$. We define $\mathbf{x}=(E_H,I_H,E_M,I_M,E_W,I_W,S_H,p)$ and we write the system as $\dot{\mathbf{x}}=\mathcal{F}(\mathbf{x})-\mathcal{V}(\mathbf{x})$, where
$$\mathcal{F}(\mathbf{x})^{\top}=\begin{pmatrix}
	\displaystyle\frac{\beta_{MH}}{H} I_M S_H +\displaystyle\frac{\beta_{WH}}{H} I_W S_H \\ 0 \\
	\displaystyle\frac{\beta_{HM}}{H} (K(1-p)-E_M-I_M) I_H \\
	0 \\
	\displaystyle\frac{\beta_{HW}}{H} (Kp-E_W-I_W) I_H \\
	0 \\
	0 \\
	0
\end{pmatrix} ,$$
and $\mathcal{V}(\mathbf{x})$ contais the rest of the terms.

The two relevant equilibria are both disease-free, one is the {\it Wolbachia}-free equilibrium, $\mathbf{x_0^M}=(0,0,0,0,0,0,H,0)$ and the other the full invasion equilibrium, $\mathbf{x_0^W}=(0,0,0,0,0,0,H,1)$.
Matrix $V$ is the same in both cases, namely
$$
V=\begin{pmatrix}
	\gamma_H+b_H & 0 & 0 & 0 & 0 & 0\\
	-\gamma_H & \sigma_H+b_H & 0 & 0 & 0 & 0 \\
	0 & 0 & \gamma_M+d_M & 0 & 0 & 0 \\
	0 & 0 & -\gamma_M & d_M & 0 & 0\\
	0 & 0 & 0 & 0 & \gamma_W+d_W & 0 \\
	0 & 0 & 0 & 0 & -\gamma_W & d_W \\
\end{pmatrix}.$$
On the other hand, $F$ has a different value at each equilibrium, namely
$$F_M=\begin{pmatrix}
	0 & 0 & 0 & \beta_{MH} & 0 & \beta_{WH}\\
	0 & 0 & 0       & 0 & 0 & 0\\
	0 & \displaystyle\frac{\beta_{HM}}{H}K & 0 & 0 & 0 & 0\\
	0 & 0 & 0       & 0 & 0 & 0\\
	0 & 0 & 0       & 0 & 0 & 0\\
	0 & 0 & 0       & 0 & 0 & 0\\
\end{pmatrix}$$ and 
$$F_W=\begin{pmatrix}
	0 & 0 & 0 & \beta_{MH} & 0 & \beta_{WH}\\
	0 & 0 & 0       & 0 & 0 & 0\\
	0 & 0 & 0 & 0 & 0 & 0\\
	0 & 0 & 0       & 0 & 0 & 0\\
	0 & \displaystyle\frac{\beta_{HW}}{H}K & 0       & 0 & 0 & 0\\
	0 & 0 & 0       & 0 & 0 & 0\\
\end{pmatrix}.$$
Although computed differently, we recover the value of $R_0^M$ obtained in \eqref{def:R0M} for system \eqref{sys:SEIR} replacing $K^*$ by $K$. The value of $R_0^W:=\rho(F_W V^{-1})$ obtained is, precisely, the value given in \eqref{def:R0W}. In this case, the per stage reproduction number $R_0^M$ is slightly higher than the per stage reproduction number for the sterile insect model due to the change of $K^*$ by $K$. For the values in Table \ref{tab:parameters} we find $R_0^M\approx 1.68$ and $R_0^W\approx 1.04$, {which give basic reproduction numbers of $\lrp{R_0^M}^2\approx 2.83$ and $\lrp{R_0^W}^2\approx 1.08$ respectively}. That means that even in a fully invaded population, outbreaks could still appear, but would have a smaller impact. Nevertheless these values should be taken with a grain of salt, since most of the parameters considered present a lot of variability in the literature.

To conclude this section, we present a result on the persistence of the disease in the system.
\begin{theorem}
	If there exists $p^*$ such that  $R_{p^*}>1$,  then the system \eqref{sys:Wolbachia_Simp} is uniformly persistent in the space of the initial conditions such that $p(t)\to p^*$, that is, there exists $\eta>0$ such that for each initial condition with $p(0)$ such that $p(t)\to p^*$ and $(E_H+I_H+E_M+I_M+E_W+I_W)(0)>0$ we have that
	$$\liminf_{t\to +\infty} (E_H+I_H+E_N+I_N+E_W+I_W)(t)>\eta.$$
	If   $R_{p^*}<1$ for each initial condition with $p(0)$ such that $p(t)\to p^*$ we have that 
$$\lim_{t\to +\infty} (E_H+I_H+E_N+I_N+E_W+I_W)(t)=0.$$
\end{theorem}

\begin{proof}
Let us fix $p^*=0$, the remaining cases can be dealt analogously. First of all we assume $R_{p^*}>1$ and apply  \cite[Theorem 1]{fonda88} in order to obtain our persistence result. The set of initial conditions which we refer to in the theorem is
$$\left\{(S_H,E_H,I_H,E_M,I_M,E_W,I_W,p)\in {\Bbb R}^7_{+0}\times[0,\theta[\right\}.$$ 
Note that it is an immediate consequence of the equations that  if one of the latent or of the infectious classes is nonempty, then it will remain always nonempty.
Moreover we know that if $p(0)<\theta$, then $p(t)\to p^*$. Hence, in order to prove persistence in our set we can consider a $0<\zeta<\theta$ and prove persistence in  
$$\left\{(S_H,E_H,I_H,E_M,I_M,E_W,I_W,p)\in {\Bbb R}^7_{+0}\times[0,\theta-\zeta]\right\}.$$ 
Notice that we assumed the human population constant and equal to $H$ and that the mosquito population satisfies a logistic growth. Taking this into account there exists a constant $\bar K>0$ such that the set
\[
\resizebox{\textwidth}{!}{${\cal K}=\left\{(S_H,E_H,I_H,E_M,I_M,E_W,I_W,p)\in {\Bbb R}^7_{+0}\times[0,\theta-\zeta]: S_H+E_H+I_H+E_M+I_M+E_W+I_W \le \bar K \right\}$
}
\]
is a  positively invariant compact set and each solution of system \eqref{sys:Wolbachia_Simp}  with initial condition in ${\Bbb R}^7_{+0}\times [0,\theta-\zeta
]$ enters in ${\cal K}$. For each $x_0=(S_H^0,E_H^0,I_H^0,E_M^0,I_M^0,E_W^0,I_W^0,\allowbreak p^0) \in {\cal K}$ there exists exactly one solution $x(t;x_0)$ of system \eqref{sys:Wolbachia_Simp}  defined in  ${\Bbb R}_{0+}$ and such that $x(0;x_0)=x_0$ and $x(t;x_0)\in  {\cal K}$ for all $t\ge 0$. We have that $x_0\to x(t;x_0)$ is a semi-dynamical system in ${\cal K}$. 

Consider the set
$${\cal S}= \left\{ (S_H,E_H,I_H,E_M,I_M,E_W,I_W,p) \in {\cal K}: E_H+I_H+E_M+I_M+E_W+I_W = 0 \right\}.$$
We have that the set ${\cal K}\setminus{\cal S}$ is invariant by the remark above about the latent and the infectious classes. 
As we have   $R_{p^*}>1$ we can consider $\delta_1>0$ and $\eta>0$ such that 

\begin{equation} 
	\begin{array}{l}
		\displaystyle \left(\frac{\gamma_M \beta_{HM}\beta_{MH}}{Hd_M(\gamma_M+d_M)}-\frac{(1+\delta_1)(\gamma_H+b_H)(\sigma_H+b_H)}{\gamma_H}\right)(K(1-p^*)-2\eta)\\
		\\
		\displaystyle +\left(\frac{\gamma_W \beta_{HW}\beta_{WH}}{Hd_W(\gamma_W+d_W)}-\frac{(1+\delta_1)(\gamma_H+b_H)(\sigma_H+b_H)}{\gamma_H}\right)  (Kp^*-2\eta)>0.
	\end{array}
	\label{Rp*}
\end{equation}

We consider $\xi$ and $\delta_2$ such that $0<\delta_2<\xi<\delta_1$ and define in ${\cal K}$ the map 
	\[
	\resizebox{\textwidth}{!}{$
		\begin{array}{l}
			P(S_H,E_H,I_H,E_M,I_M,E_W,I_W)=\\
			\\
			(1+\xi) E_H+\frac{(1+\delta_1)(\gamma_H+b_H)}{\gamma_H}I_H+\frac{\gamma_M\beta_{MH}}{d_M(\gamma_M+d_M)}E_M+\frac{(1+\delta_2)\beta_{MH}}{d_M}I_M+\frac{\gamma_W\beta_{WH}}{d_W(\gamma_W+d_W)}E_W+\frac{(1+\delta_2)\beta_{WH}}{d_W}I_W.
		\end{array}$}
	\]
	
Let us consider also for sufficiently small $\varepsilon$ the neighbourhood of ${\cal S}$ 
$${\cal U}=\left\{x\in  {\cal K}:\,\,P(x)<\varepsilon\right\}.$$
We have that 
$$P(x)=0\Longleftrightarrow x\in S.$$
Moreover, let us assume, in order to arrive to a contradiction, that: 
\begin{equation}
	\exists x_0\in {\cal U}\setminus {\cal S}  \mbox{ such that } P(x(t;x_0))<\varepsilon \mbox{ for all } t>0. 
	\label{epsilon}
\end{equation} Let $\phi(t)=P(x(t;x_0))$, we are going to prove that there exists $k>0$ such that 
\begin{equation}\phi'(t)\ge k \phi(t)\label{contra}
\end{equation}
for large $t$.
In fact, taking into account \eqref{epsilon}, we obtain that there exists $\varepsilon^*>0$ such that $\liminf_{t\to +\infty} S_H(t)>\frac{b_H H}{\varepsilon^*+b_H}$ and this $\varepsilon^*>0$ can be chosen sufficiently small if we choose $\varepsilon$ small. We assume that $\varepsilon$ is chosen in order to imply that  $\displaystyle \frac{\varepsilon^*+b_H}{b_H}(1+\delta_2)<1+\xi$ and also that the  latent and infected mosquitoes classes are smaller than $\eta$ for $t>0$ (this will be useful after and is possible by  \eqref{epsilon}). Then, we evaluate $\phi'(t)$ and recall that $p(t)\to p^*$ when $t\to +\infty$.
We obtain
	\begin{eqnarray*}
		\phi'(t)&=& \left(\frac{b_H}{\varepsilon^*+b_H}(1+\xi)-(1+\delta_2)\right)(\beta_{MH} I_M+\beta_{WH} I_W)\\
		&+& (\delta_1-\xi)(\gamma_H+b_H)E_H+\delta_2\left(\frac{\beta_{MH}\gamma_M}{d_M}E_M+\frac{\beta_{WH}\gamma_W}{d_W}E_W\right)\\
		&+&\left( \left(\frac{\gamma_M \beta_{HM}\beta_{MH}}{Hd_M(\gamma_M+d_M)}-\frac{(1+\delta_1)(\gamma_H+b_H)(\sigma_H+b_H)}{\gamma_H}\right)(K(1-p)-E_M-I_M)\right.\\
		&+& \left.\left(\frac{\gamma_W \beta_{HW}\beta_{WH}}{Hd_W(\gamma_W+d_W)}-\frac{(1+\delta_1)(\gamma_H+b_H)(\sigma_H+b_H)}{\gamma_H}\right)(Kp-E_W-I_W)\right)I_H.
	\end{eqnarray*}
We have that $p(t)\to p^*$ and hence, by \eqref{Rp*}, we have that for sufficiently large $t$ the coefficient of $I_H$ in the last expression is positive.
The existence of $k>0$ satisfying \eqref{contra} follows and this contradicts \eqref{epsilon}. We conclude that  ${\cal S}$ is an uniform repeller and the result for $R_{p^*}>1$ follows.

The case $R_{p^*}<1$ can be obtained in the spirit of the previous one constructing this time a function $\phi^*$ for which there exists $k^*<0$ such that for each $t>0$ 
\begin{equation}{\phi^*}'(t)\le -k^* \phi^*(t).
\end{equation}

\end{proof}

\section{Control Problem and Impulsive Dynamics}\label{sec:control}
We place ourselves in the case of a dengue outbreak in a fully susceptible population. In this work, we will consider the case where the goal of the releases is to minimize the number of dengue cases (among humans) during the duration of the outbreak. Therefore, considering a time window of size $T$, we want to find $u$ minimizing $\int_0^T I_H(t)\, dt$. Other works have studied related problems in the case of {\it Wolbachia} \cite{zhang_lui}, or problems involving only the mosquito population \cite{ABP,cont_ster_wolb,radon,JDE} considering controls in $ L^{\infty}(0,T)$. 

Field releases are done with a certain periodicity and last a short amount of time with respect to the time window considered, this leads us to consider the particular case where the control, denoted $u(\cdot)$, is assumed to be a linear combination of pulses, namely 
$$
u(t)=\sum_{i=1}^{n}c_i\delta(t-t_i).
$$
where $\delta(t)$ is the Dirac measure at $t=0$ and $0\leq t_1\leq \dots\leq t_n$ are the release times.
It is natural to impose some constraints on the control function. Usually it is assumed that the rate at which mosquitoes are released is bounded ($u\in L^{\infty}(0,T)$) but also that the total amount of mosquitoes used is bounded ($\int_0^T u(t)dt \leq C$). Our approach is different. We also assume that we have a limited amount of mosquitoes at our disposal, $C$, but we assume that all of them are used. Since our control function is a linear combination of pulses, this translates into imposing the constraint $\sum_{i=1}^{n}c_i=C$. Therefore, for both systems \eqref{sys:Sterile} and \eqref{sys:Wolbachia_Simp}, the optimization problem we will study is 
\begin{equation}\label{prob:IH}\tag{\ensuremath{\mathcal{P}}}
\begin{minipage}{10cm}
\it Minimize $J(u)$ over the set of time jumps $(t_i)_{1\leq i\leq n}\in[0,T]^n$ and of nonnegative coefficients $(c_i)_{1\leq i\leq n}$ such that $\displaystyle \sum_{i=1}^n c_i= C$,
\end{minipage}
\end{equation}
where the number of jumps, $n$, and the time horizon, $T$, are fixed and the cost functional $J$, is given by
\begin{equation}\label{def:JIH}
J(u):=\int_0^T I_H(t)dt,
\end{equation}
and is proportional to the average number of humans that are in the infected compartment each day between $0$ and $T$.
Since we are going to deal with several jumps it is convenient to introduce some notation first. We consider $n$ jumps performed at times $t_i$, for $i=1,\dots,n$. If needed, for the sake of notational simplicity, we will denote $t_0=0$ and $t_{n+1}=T$. Since functions may present discontinuities we introduce the notations
\begin{equation*}
F(t_i^{-}) :=  \lim\limits_{t\to t_i^{-}} F(t),\qquad
F(t_i^+) := \lim\limits_{t\to t_i^+} F(t),
\end{equation*}
where $F(t)$ represents any function. The equations for $M_S$ and $p$ in systems \eqref{sys:Sterile} and \eqref{sys:Wolbachia_Simp} must be adapted to the impulsive formulation of the problem. We also introduce the characteristic function of a set $\mathcal{S}$, equal to $1$ when its variable belongs to $\mathcal{S}$ and $0$ elsewhere. In what follows, we will denote it $\charac{\mathcal{S}}$.

By considering $u$ defined by $u(t)=\sum_{i=1}^{n}c_i\delta(t-t_i)$ in systems \eqref{sys:Wolbachia_Simp} and \eqref{sys:Sterile} we can pass from a infinite dimensional optimization problem to a discrete one. Here we detail how, by doing this passage, these systems where the control appears become differential equations with jump discontinuities. In order to do so we consider $u$ given by  $u(t)=\sum_{i=1}^{n}\frac{c_i}{\varepsilon}\charac{[t_i,t_i+\varepsilon]}$ and we take the limit $\varepsilon\to 0$. The following proof is adapted from \cite{delta_jump}. We detail the deduction of equation \eqref{eq:jump_p}. However, equation \eqref{eq:jump_Ms} can be easily obtained following the same reasoning.

\begin{proposition}\label{prop:jump}
	Let us consider $p_{\varepsilon}$, solving the following equation 
	$$
	\left\{\begin{array}{ll}
		p_{\varepsilon}'(t)=f(p_{\varepsilon}(t))+\frac{c_i}{\varepsilon}\charac{[t_i,t_i+\varepsilon]} g(p_{\varepsilon}(t)), & t\in[t_{i-1},t_{i+1}] \\
		p_{\varepsilon}(t_{i-1})=p_{i-1}. &\end{array}\right.
	$$
	Let $G$ be the antiderivative vanishing at zero of $1/g(p)$, that is $G(p):=\int_0^p \frac{dq}{g(q)}$. Then, when $\varepsilon$ tends to $0$, $p_{\varepsilon}(\cdot)$ converges pointwise to $p(\cdot)$ given by
	$$
	p(t) = \begin{cases} p^{-}(t) ,& t\in[t_{i-1},t_i] \\
		p^{+}(t) ,& t\in(t_i,t_{i+1}]      
	\end{cases}
	$$  
	where $p^-$ and $p^+$ solve
	$$
	\left\{\begin{array}{ll}
		\frac{dp^{-}}{dt}(t) = f(p^{-}(t)) & \\
		p^{-}(t_{i-1}) = p_{i-1}, &
	\end{array}\right.
	\quad\text{and}\quad
	\left\{\begin{array}{ll}
		\frac{dp^{+}}{dt}(t) = f(p^{+}(t)) & \\
		p^{+}(t_{i}) = G^{-1}(G(p^{-}(t_i))+c_i), &
	\end{array}\right.
	$$  
	respectively.
	
\end{proposition}

\begin{proof}
	Outside the interval $[t_i,t_i+\varepsilon]$ the behaviour of $p(t)$ is clear. We study the behaviour of $p_{\varepsilon}(t)$ in $[t_i,t_i+\varepsilon]$, in order to establish the jump of $p(t)$ at $t_i$. Let
	
	$$p_{\varepsilon}(t)=p^{-}(t_i)+\int_{t_{i}}^{t}f(p_\varepsilon(s))+\frac{c_i}{\varepsilon}g(p_\varepsilon(s))ds.$$
Then, for every $t\in [t_i,t_i+\varepsilon]$, one has
	\begin{eqnarray*}
		\left|p_{\varepsilon}(t)\right| &\leq& \left|p^{-}(t_i)\right|+\int_{t_{i}}^{t}\left|f(0)\right|+\frac{c_i}{\varepsilon}\left|g(0)\right|ds+\int_{t_{i}}^{t}\left(L_f+\frac{c_i}{\varepsilon}L_g\right)\left|p_{\varepsilon}(s)\right|ds \\
		&\leq& 1+\frac{c_i}{K}+\int_{t_{i}}^{t}\left(L_f+\frac{c_i}{\varepsilon}L_g\right)\left|p_{\varepsilon}(s)\right|ds
	\end{eqnarray*}
	where $L_f$ and $L_g$ are the Lipschitz constants of $f(\cdot)$ and $g(\cdot)$ respectively. These constants exist since both functions are $C^1$ in $[0,1]$.
	Using Grönwall's inequality we obtain that
	$$\left|p_{\varepsilon}(t)\right| \leq \lrp{1+\frac{c_i}{K}}\exp{\left(\varepsilon L_f+c_i L_g\right)},$$
	which is bounded.
	Let us consider now $z_{\varepsilon}$, the solution of 
	$$
	\left\{\begin{array}{ll}
		z_{\varepsilon}'(t) = \frac{c_i}{\varepsilon} g(z_{\varepsilon}(t)) &\\
		z_{\varepsilon}(t_{i})= p(t_i^-), &
	\end{array}\right.
	$$
	We prove now that, in the limit, both $z_{\varepsilon}$ and $p_{\varepsilon}$ present the same jump at $t_i$. In order to do this we compute for $t\in[t_i,t_{i}+\varepsilon]$,
	\begin{eqnarray*}
		\left|z_{\varepsilon}(t)-p_{\varepsilon}(t)\right| &\leq& \int_{t_{i}}^{t}\left|f(p_{\varepsilon}(s))\right|ds+\int_{t_{i}}^{t}\frac{c_i}{\varepsilon}\left|g(z_{\varepsilon}(s))-g(p_{\varepsilon}(s))\right|ds \\
		&\leq& \varepsilon M_f+\int_{t_{i}}^{t}\frac{c_i}{\varepsilon}L_g\left|z_{\varepsilon}(s)-p_{\varepsilon}(s)\right|ds\\
	\end{eqnarray*}
	where $M_f=\max_{p\in[0,1]}f(p)$. Using again Grönwall's Lemma we obtain
	$$\left|z_{\varepsilon}(t)-p_{\varepsilon}(t)\right|  \leq \varepsilon M_f\exp{\left(  c_i L_g\right)}\to 0 \mbox{ as }
	\varepsilon\to 0.$$
	This proves that $\sup_{t\in[t_i,t_i+\varepsilon]}\left|z_{\varepsilon}(t)-p_{\varepsilon}(t)\right|\to 0$ when $\varepsilon\to 0$, and therefore $z_{\varepsilon}$ and $p_{\varepsilon}$ present the same jump at $t_i$ in the limit.
	To conclude, we solve $z_{\varepsilon}(t)$ in $[t_i,t_i+\varepsilon]$, $$\int_{t_i}^{t_i+\varepsilon}\frac{z'(s)}{g(z_{\varepsilon}(s))}ds=\int_{t_i}^{t_i+\varepsilon}\frac{c_i}{\varepsilon}ds= c_i,$$
	which leads to $G(z_{\varepsilon}(t_i+\varepsilon))-G(z_{\varepsilon}(t_i))=c_i$ and thus $z_{\varepsilon}(t_i+\varepsilon)=G^{-1}(G(z_{\varepsilon}(t_i))+c_i)$. Taking the limit $\varepsilon\to 0$ we conclude that $p^+(t_i)=G^{-1}(G(p^{-}(t_i))+c_i)$.
	
\end{proof}

\subsection{Sterile insect technique}
In order to find the equation satisfied by $M_S$ we take
$$
u(t)=\sum_{i=1}^{n}\frac{c_i}{\varepsilon}\charac{[t_i,t_i+\varepsilon]}.
$$
Hence, the equation satisfied by $M_S'$ becomes 
$$
M_S'(t)=\sum_{i=1}^{n}\frac{c_i}{\varepsilon}\charac{[t_i,t_i+\varepsilon]}-d_S M_S(t).
$$
Taking the limit $\varepsilon\to 0$ we obtain that the equation converges to
\begin{equation}
\label{eq:jump_Ms}
\left\{\begin{array}{lll}
M'_S(t) & = & -d_S M_S(t), \quad t\in[t_i,t_{i+1}], \quad i=0,\dots,n\\
M_S(t_i^+) & = &  M_S(t_i^{-})+c_i, \quad i=1,\dots,n
\end{array}  \right.\\
\end{equation}
We can solve this equation explicitly. Since the initial condition is $M_S(0)=0$ the solution reads 
\begin{equation}\label{eq:Ms}
M_S(t)=\sum\limits_{j=1}^{i} c_j e^{-d_S(t-t_j)}, \quad t\in[t_i,t_{i+1}], \quad i=1,\dots,n
\end{equation}

\subsection{{\it Wolbachia} method}

Looking at the equation on $p$ in system \eqref{sys:Wolbachia_Simp} and substituting the control function by 
$$
u(t)=\sum_{i=1}^{n}\frac{c_i}{\varepsilon}\charac{[t_i,t_i+\varepsilon]},
$$ 
we obtain
\begin{equation}
\label{eq:p_imp_eps}
p'(t)=f(p(t))+g(p(t))\sum_{i=1}^{n}\frac{c_i}{\varepsilon}\charac{[t_i,t_i+\varepsilon]}.
\end{equation}

Let $G$ be the antiderivative vanishing at zero of $1/g(p)$, that is $G(p):=\int_0^p \frac{dq}{g(q)}$. When we take the limit $\varepsilon\to 0$ in equation \eqref{eq:p_imp_eps}, we obtain:

\begin{equation}
\label{eq:jump_p}
\left\{\begin{array}{lll}
p'(t) & = & f(p(t)), \quad t\in[t_i,t_{i+1}], \quad i=0,\dots,n\\
p(t_i^+) & = & G^{-1}(G(p(t_i^-))+c_i) , \quad i=1,\dots,n
\end{array}  \right.\\
\end{equation}

\section{Optimality conditions}\label{sec:optcond}
We devote this section to the computation of the gradients of the functional $J$ in problem \eqref{prob:IH} for systems \eqref{sys:Sterile} and \eqref{sys:Wolbachia_Simp}. These gradients will be used in the numerical simulations of section \ref{sec:results}. We start by discussing things in a general setting and, later we apply to our problems.

Let $\mathbf{X}:\mathbb{R}^+\to\mathbb{R}^N$ be the solution to 
\begin{equation}\label{eq:X}
	\left\{\begin{array}{ll}
		\mathbf{X}'(t)=A(\mathbf{X}(t))+B(\mathbf{X}(t))y(t), & t\in[0,T] \\
		\mathbf{X}(0)=\mathbf{X_0}, &\end{array}\right.
\end{equation}
with $A,B:\mathbb{R}^+\to\mathbb{R}^N$ continuous and  $y:\mathbb{R}^+\to \mathbb{R}$ the solution to the differential equation with jump discontinuities 
$$
\left\{\begin{array}{ll}
	y'(t)=a(y(t)), & t\in[0,T] \\
	y(t_i^+)=b(y(t_i^-),c_i), & i=1,\dots,n 
	
\end{array}\right.
$$
with $a,b:\mathbb{R}^+\to\mathbb{R}$.
Now, consider $y_{\varepsilon}(t)$, the solution to 
$$
\left\{\begin{array}{ll}
	y'_{\varepsilon}(t)=a(y_{\varepsilon}(t)), & t\in[0,T] \\
	y_{\varepsilon}(t_i^+)=b(y_{\varepsilon}(t_i^-),c_i), & i\neq k \\
	y_{\varepsilon}(\tilde{t}_k^+)=b(y_{\varepsilon}(\tilde{t}_k^-),c_k),
\end{array}\right.
$$ where $\tilde{t}_k=t_k+\varepsilon$. Finally, lets consider also $\mathbf{X}_{\varepsilon}$ the solution to 
$$
\left\{\begin{array}{ll}
	\mathbf{X}_{\varepsilon}'(t)=A(\mathbf{X}_{\varepsilon}(t))+B(\mathbf{X}_{\varepsilon}(t))y_{\varepsilon}(t), & t\in[0,T] \\
	\mathbf{X}_{\varepsilon}(0)=\mathbf{X_0}, &\end{array}\right.
$$

We consider $\mathbf{X}$ to be a function of time, nevertheless, the value of the parameters $t_i$ and $c_i$, $i=1,\dots,n$ affects the value of $\mathbf{X}(t)$ for any $t>t_1$. In general, and for the rest of the section, we define the variation of any given function of time, $F_{\chi}$, depending on a parameter, $\chi$, as

\begin{equation*}
	\delta_{\chi}F_{\chi}(t):=\lim\limits_{\varepsilon\to 0} \frac{F_{\chi+\varepsilon}(t)-F_{\chi}(t)}{\varepsilon}.
\end{equation*}
As an example, in the case of the variation of $\mathbf{X}$ with respect to a given $t_k$ we have

\begin{equation*}
	\delta_{t_k}\mathbf{X}(t):=\lim\limits_{\varepsilon\to 0} \frac{\mathbf{X}_{\varepsilon}(t)-\mathbf{X}(t)}{\varepsilon}.
\end{equation*}

From equation \eqref{eq:X} we have that $$\mathbf{X}(t)=\mathbf{X}_0+\int_0^t A(\mathbf{X}(s))ds+\int_0^t B(\mathbf{X}(s))y(s)ds.$$
For a given $k\in\enum{1}{n}$, in case $t<t_k$, one has $\delta_{t_k}\mathbf{X}(t)=0$, since the time of the jump has no effect until it occurs.  In case $t>t_k$
\begin{eqnarray*}
	\delta_{t_k}\mathbf{X}(t)&=&\delta_{t_k}\int_0^t A(\mathbf{X}(s))ds+ \delta_{t_k}\lrp{\int_0^{t_k} B(\mathbf{X}(s))y(s)ds+\int_{t_k}^{t} B(\mathbf{X}(s))y(s)ds} \\
	&=& \int_0^t \delta_{t_k} A(\mathbf{X}(s))ds+ B(\mathbf{X}(t_k))y(t_k^-)- B(\mathbf{X}(t_k))y(t_k^+)+\int_{0}^{t} \delta_{t_k}\lrp{B(\mathbf{X}(s))y(s)}ds \\
	&=&  \int_0^t \lrp{\mathbf{D} A(\mathbf{X}(s))+\mathbf{D} B(\mathbf{X}(s))y(s)}\delta_{t_k}\mathbf{X}(s)ds+ B(\mathbf{X}(t_k))(y(t_k^-)-y(t_k^+)) \\
	&& + \int_{0}^{t} B(\mathbf{X}(s))\delta_{t_k}y(s) ds.
\end{eqnarray*}
We can express this as an ordinary differential equation with a jump discontinuity:
$$
\left\{\begin{array}{ll}
	\lrp{\delta_{t_k}\mathbf{X}}'(t)=\lrp{\mathbf{D} A(\mathbf{X}(t))+\mathbf{D} B(\mathbf{X}(t))y(t)}\delta_{t_k}\mathbf{X}(t)+B(\mathbf{X}(t))\delta_{t_k}y(t), & t\in[0,T] \\
	\delta_{t_k}\mathbf{X}(0)=0,\\
	\delta_{t_k}\mathbf{X}(t_k^+)=\delta_{t_k}\mathbf{X}(t_k^-)+B(\mathbf{X}(t_k))(y(t_k^-)-y(t_k^+)).
	&\end{array}\right.
$$

But since  $\delta_{t_k}\mathbf{X}(t)=0$ for $t<t_k$, we can simplify this system to:

\begin{equation}\label{eq:delta_tk_X}
	\resizebox{0.87\textwidth}{!}{$
		\begin{cases}
			\lrp{\delta_{t_k}\mathbf{X}}'(t)=\lrp{\mathbf{D} A(\mathbf{X}(t))+\mathbf{D} B(\mathbf{X}(t))y(t)}\delta_{t_k}\mathbf{X}(t)+B(\mathbf{X}(t))\delta_{t_k}y(t), & t\in[t_k,T] \\
			\delta_{t_k}\mathbf{X}(t_k^+)=B(\mathbf{X}(t_k))(y(t_k^-)-y(t_k^+)),
			&\end{cases}$}
\end{equation}
where $\delta_{t_k}y(t):=\lim_{\varepsilon\to 0}(y_{\varepsilon}(t)-y(t))/\varepsilon$.

Following the same lines we consider now $y_{\varepsilon}(t)$ as the solution to
$$
\left\{\begin{array}{ll}
	y'_{\varepsilon}(t)=a(y(t)_{\varepsilon}), & t\in[0,T] \\
	y_{\varepsilon}(t_i^+)=b(y_{\varepsilon}(t_i^-),c_i), & i\neq k, \\
	y_{\varepsilon}(t_k^+)=b(y_{\varepsilon}(t_k^-),c_k+\varepsilon).
\end{array}\right.
$$
In this case, for $t>t_k$ we have
\begin{eqnarray*}
	\delta_{c_k}\mathbf{X}(t)&=&\int_0^t \delta_{c_k}A(\mathbf{X}(s)) + \delta_{c_k}B(\mathbf{X}(s))y(s)+B(\mathbf{X}(s))\delta_{c_k}y(s)\, ds \\
	&=&  \int_0^t \lrp{\mathbf{D} A(\mathbf{X}(s))+\mathbf{D} B(\mathbf{X}(s))y(s)}\delta_{c_k}\mathbf{X}(s)+B(\mathbf{X}(s))\delta_{c_k}y(s) \, ds.
\end{eqnarray*}
Since  $\delta_{c_k}\mathbf{X}(t)=0$ for $t<t_k$, we can express this as the following ordinary differential equation:

\begin{equation}\label{eq:delta_ck_X}
	\resizebox{0.87\textwidth}{!}{$
		\left\{\begin{array}{ll}
			\lrp{\delta_{c_k}\mathbf{X}}'(t)=\lrp{\mathbf{D} A(\mathbf{X}(t))+\mathbf{D} B(\mathbf{X}(t))y(t)}\delta_{c_k}\mathbf{X}(t)+B(\mathbf{X}(t))\delta_{c_k}y(t), & t\in[t_k,T] \\
			\delta_{c_k}\mathbf{X}(t_k^+)=0.
			&\end{array}\right.$}
\end{equation}
with, again, $\delta_{c_k}y(t):=\lim_{\varepsilon\to 0}(y_{\varepsilon}(t)-y(t))/\varepsilon$.

In problem \eqref{prob:IH}, the functional we want to minimize is $J(u)=\int_0^T I_H(t)dt$. Since $I_H(t)$ is continuous we have that $$\delta_{t_k}J(u)=\int_0^T \delta_{t_k}I_H(t)dt,$$
we also have that $\delta_{c_k}J(u)=\int_0^T \delta_{c_k}I_H(t)dt$.
Hereafter we use expressions \eqref{eq:delta_tk_X} and \eqref{eq:delta_ck_X} in order to compute $\delta_{t_k}J$ and $\delta_{c_k}J$ for systems $\eqref{sys:Sterile}$ and $\eqref{sys:Wolbachia_Simp}$.

\subsection{Sterile Insect Technique}
We consider system \eqref{sys:Sterile}. The variable satisfying a differential equation with a jump discontinuity is $M_S(t)$. Therefore, considering $\mathbf{X}(t)=\lrp{S_H(t),E_H(t),I_H(t),S_M(t),E_M(t),I_M(t)}$ and $y(t)=M_S(t)$ we find that $\delta_{t_k}J=\int_{t_k}^{T}\lrp{\delta_{t_k}\mathbf{X}(t)}_{3}dt$
and $\delta_{c_k}J=\int_{t_k}^{T}\lrp{\delta_{c_k}\mathbf{X}(t)}_{3}dt$ where $\delta_{t_k}\mathbf{X}(t)$ and $\delta_{c_k}\mathbf{X}(t)$ are defined by equations \eqref{eq:delta_tk_X} and \eqref{eq:delta_ck_X} respectively and the subscript stands for the third component of the vector.
There are nonetheless two more terms to compute, $\delta_{t_k}M_S(t)$ and $\delta_{c_k}M_S(t)$. In the case of the Sterile Insect Technique we have a closed expression for $M_S(t)$, see equation \eqref{eq:Ms}, therefore the computation of the variation of $J$ with respect to $t_k$ and $c_k$ is straightforward. We have
\begin{equation}
	\label{eq:delta_tk_MS}
	\delta_{t_k}M_S(t)=\begin{cases}
		0 , \quad t\in[0,t_k] \\
		d_Sc_ke^{-d_S(t-t_k)} \quad t\in(t_k, T],
	\end{cases}    
\end{equation}
and
\begin{equation}
	\label{eq:delta_ck_MS}
	\delta_{c_k}M_S(t)=\begin{cases}
		0 , \quad t\in[0,t_k] \\
		e^{-d_S(t-t_k)} \quad t\in(t_k, T].
	\end{cases}    
\end{equation}

\subsection{{\it Wolbachia} method}
In the case of the use of {\it Wolbachia} (system \eqref{sys:Wolbachia_Simp}) the variable satisfying a differential equation with a jump discontinuity is the proportion of {\it Wolbachia} infected mosquitoes, $p(t)$. We consider now $$\mathbf{X}(t)=\lrp{S_H(t),E_H(t), I_H(t),E_M(t),I_M(t),E_W(t),I_W(t)}$$ and $y(t)=p(t)$. Once more, $\delta_{t_k}J=\int_{t_k}^{T}\lrp{\delta_{t_k}\mathbf{X}(t)}_{3}dt$
and $\delta_{c_k}J=\int_{t_k}^{T}\lrp{\delta_{c_k}\mathbf{X}(t)}_{3}dt$. Since the expressions of $\delta_{t_k}p(t)$ and $\delta_{c_k}p(t)$ are significantly harder to find than in the sterile insect case we compute them in the following propositions.

\begin{proposition}\label{prop:deltatk}
	Let $p$ solve 
	$$
	\left\{\begin{array}{ll}
		p'(t)=f(p(t))+\sum\limits_{i=1}^{n}c_i\delta(t-t_i)g(p(t)), & t\in[0,T] \\
		p(0)=p_{0}. &\end{array}\right.
	$$
	with $p(t_i^+)\neq\theta$ for all $i=1,\dots,n$.
	Let $c_i$ be fixed for all $i=1,\dots,n$ and let $p_{\varepsilon}(t)$ solve
	$$
	\left\{\begin{array}{ll}
		p'_{\varepsilon}(t)=f(p_{\varepsilon}(t))+\sum\limits_{\substack{i=1\\ i\neq k}}^{n} c_i\delta(t-t_i)g(p(t))+c_{k}\delta(t-(t_{k}+\varepsilon))g(p_{\varepsilon}(t)), & \\
		p_{\varepsilon}(0)=p_{0}. &\end{array}\right.
	$$
	Then, the variation of $p(t)$ with respect to $t_k$, $\delta_{t_k}p(t):=\lim_{\varepsilon\to 0}\frac{p_{\varepsilon}(T)-p(T)}{\varepsilon}$, is 
	\begin{equation} 
		\resizebox{0.87\textwidth}{!}{$
			\label{eq:var_p_tk}
			\delta_{t_k}p(t)=\begin{cases}
				0 , \quad t\in[0,t_k] \\
				\frac{f(p(t_k^-))g(p(t_k^+))-f(p(t_k^+))g(p(t_k^-))}{g(p(t_k^-))}\frac{f(p(t))}{f(p(t_{i}^+))}\prod\limits_{j=k+1}^{i}{\frac{g(p(t_j^+))}{g(p(t_j^-))}\frac{f(p(t_j^-))}{f(p(t_{j-1}^+))}}, t\in(t_i, t_{i+1}], k\leq i\leq n.
			\end{cases}$}
	\end{equation}

\end{proposition}

\begin{proof}
	We begin considering $t\in [t_i,t_{i+1}]$. In each one of these intervals we have that $p'(t)=f(p(t))$. Since $f$ is bistable, $f(p)<0$ in $(0,\theta)$ and $f(p)>0$ in $(\theta,1)$. Therefore, since we assumed $p(t_i^+)\neq\theta$ for all $i=1,\dots,n$, $p(t)$ is injective in $[t_i,t_{i+1}]$, and we can write
	$$\int_{p(t_i^+)}^{p(t)}\frac{dq}{f(q)}=t-t_i.$$ We define $F$ to be the antiderivative of $1/f$ vanishing at $p(t_i^+)$, that is $F(p):=\int_{p(t_i^+)}^{p}\frac{dq}{f(q)}$, thus we obtain the relationship 
	\begin{equation}\label{eq:F}
		F(p(t))-F(p(t_i^+))=t-t_i.
	\end{equation}
	
	We remark that $p(t)=p_{\varepsilon}(t)$ for all $t\in[0,t_k]$. Therefore in that interval $\delta_{t_k}p(t)=0$. Hence, we can restrict ourselves to the case $k\leq i \leq n$. Differentiating implicitly this equation, we get
	$$
	\frac{1}{f(p(t))}\delta_{t_k}p(t)-\frac{1}{f(p(t_i^+))}\delta_{t_k}p(t_i^+)=0
	$$
	and thus 
	$$
	\delta_{t_k}p(t)=\frac{f(p(t))}{f(p(t_i^+))}\delta_{t_k}p(t_i^+).
	$$
	To compute $\delta_{t_k}p(t_i^+)$ we use that $p(t_i^+)=G^{-1}(G(p(t_i^-))+c_i)$, therefore 
	$$
	\delta_{t_k}p(t_i^+)=(G^{-1})'(G(p(t_i^-))+c_i)G'(p(t_i^-))\delta_{t_k}p(t_i^-)=\frac{g(p(t_i^+))}{g(p(t_i^-))}\delta_{t_k}p(t_i^-)
	$$
	where we used the inverse function theorem to write $(G^{-1})'=1/(G'\circ G^{-1})$. Analogously to equation \eqref{eq:F} we find that $F(p(t^-_i))-F(p(t_{i-1}^+))=t-t_{i-1}$, so $\delta_{t_k}p(t_i^-)=\frac{f(p(t^-_{i}))}{f(p(t_{i-1}^+))}\delta_{t_k}p(t_{i-1}^+)$. We can repeat this process iteratively until we get to $F(p(t^-_{k+1}))-F(p(t_{k}^+))=t-t_{k}$, then 
	$$
	\frac{1}{f(p(t^-_{k+1}))}\delta_{t_k}p(t^-_{k+1})=-1+\frac{1}{f(p(t_k^+))}\delta_{t_k}p(t_k^+)=-1+\frac{1}{f(p(t_k^+))}\frac{g(p(t_k^+))}{g(p(t_k^-))}\delta_{t_k}p(t_k^-)
	$$
	and $\delta_{t_k}p(t^-_k)=\delta_{t_k}\int_{t_{k-1}}^{t_k}(f(p(t)))dt=f(p(t^-_k))$ from which we can deduce the final expression.
	
\end{proof}
Note that in the expression of $\delta_{t_k}p(t)$ we are using the convention that if the productory subscript is bigger than the superscript, then its equal to $1$.

\begin{proposition}\label{prop:deltack}
	Let $p$ solve 
	$$
	\left\{\begin{array}{ll}
		p'(t)=f(p(t))+\sum_{i=1}^{n}c_i\delta(t-t_i)g(p(t)), & t\in[0,T] \\
		p(0)=p_{0}. &\end{array}\right.
	$$
	with $p(t_i^+)\neq\theta$ for all $i=1,\dots,n$.
	Let $t_i$ be fixed for all $i=1,\dots,n$ and let $p_{\varepsilon}(t)$ solve
	$$
	\left\{\begin{array}{ll}
		p'_{\varepsilon}(t)=f(p_{\varepsilon}(t))+\sum\limits_{\substack{i=1\\ i\neq k}}^{n} c_i\delta(t-t_i)g(p(t))+(c_{k}+\varepsilon)\delta(t-t_{k})g(p_{\varepsilon}(t)), & \\
		p_{\varepsilon}(0)=p_{0}. &\end{array}\right.
	$$
	Then, the variation of $p(t)$ with respect to $c_k$, $\delta_{c_k}p(t):=\lim_{\varepsilon\to 0}\frac{p_{\varepsilon}(t)-p(t)}{\varepsilon}$, is 
	\begin{equation} 
		\resizebox{0.87\textwidth}{!}{$
			\label{optcond:pck}
			\delta_{c_k}p(t)=\begin{cases}
				0 , \quad t\in[0,t_k] \\
				g(p(t_k^+))\frac{f(p(t))}{f(p(t_{i}^+))}\prod\limits_{j=k+1}^{i}{\frac{g(p(t_j^+))}{g(p(t_j^-))}\frac{f(p(t_j^-))}{f(p(t_{j-1}^+))}}, \quad t\in(t_i, t_{i+1}], k\leq i\leq n.
			\end{cases}$}
	\end{equation}

\end{proposition}

\begin{proof}
	Following a very similar process to the one carried out in the proof of Proposition~\ref{prop:deltatk}, we obtain $\delta_{c_k}p(t)=\frac{f(p(t))}{f(p(t_i^+))}\delta_{c_k}p(t_i^+)$. In problem \eqref{prob:IH}, the $c_i$ must satisfy the constraint $\sum_{i=1}^{n}c_i=C$, but we are not dealing with this constraint for the moment, therefore $\delta_{c_k}c_i=\delta_{ki}$, where $\delta_{ki}$ is the Kronecker's delta.
	
	We compute $\delta_{c_k}p(t_i^+)$, obtaining
	\begin{eqnarray*}
		\delta_{c_k}p(t_i^+)&=&(G^{-1})'(G(p(t_i^-))+c_i)\lrp{G'(p(t_i^-))\delta_{c_k}p(t_i^-)+\delta_{ki}}\\
		&=&\frac{g(p(t_i^+))}{g(p(t_i^-))}\delta_{c_k}p(t_i^-)+\delta_{ki}g(p(t_i^+)).
	\end{eqnarray*}
	Following the same lines of the proof of Proposition \ref{prop:deltatk}, from equation \eqref{eq:F} applied in the interval $[t_{i-1},t_i]$, differentiating implicitly we obtain $$\delta_{c_k}p(t_i^-)=\frac{f(p(t_i^-))}{f(p(t_{i-1}^+))}\delta_{c_k}p(t_{i-1}^+).$$ Finally, iterating the process until the interval $[t_k,t_{k+1}]$ and rearranging the terms we obtain the result.
	
\end{proof}

\section{Results}\label{sec:results}
We present in this section the optimal solutions of problem \eqref{prob:IH}, obtained using numerical simulations. We optimize simultaneously the time profile of the releases and the amount of mosquitoes released in each one. We allow two releases to occur at the same time. This implies that at that time a release with the total amount of mosquitoes of the two releases combined is done, reducing at each time this occurs the number of effective releases by one. The simulations have been performed using \textit{Python}.  
For the numerical optimization, the time variables are updated by using a standard step variable gradient descent method. Regarding the weights $(c_i)_{1\leq i\leq n}$, due to the constraint $\sum_{i=1}^{n}c_i=C$, we used an augmented Lagrangian algorithm. An explanation of the method used can be found in section \ref{sec:Uzawa}. The details about the computation of the gradients of the functional are detailed in section \ref{sec:optcond}. The models considered in this work capture the essence of the interaction of the modified vectors with the disease and its effect on the transmission. Nevertheless, in order to be precise, more complex models should be considered. These simulations do not intend to give quantitative results, but rather qualitative ones which, nevertheless, allow us to explore certain scenarios that go beyond what is done in field interventions nowadays.

To model the start of an outbreak we place ourselves in the context of a nearly fully susceptible population where a small number of infected humans and mosquitoes  are present. Since in our model the total amount of humans is constant and since we consider at $t=0$ the mosquito compartment at equilibrium, we need to subtract this initial amount of infectious from the respective susceptible compartments. Thus, the initial conditions for our simulations will be $$(S_H(0),E_H(0),I_H(0),S_M(0),E_M(0),I_M(0))=(H-I_H^{0},0,I_H^0,K^*-I_M^0,0,I_M^0),$$ with $I_H^0\ll H$ and $I_M^0\ll K^*$  (in particular for the simulations we chose $I_H^0=I_M^0=20$). All the other variables in the two systems are set to $0$ at the start, namely $$M_S(0)=0 \quad \mbox{and} \quad (E_W(0),I_W(0),p(0))=(0,0,0).$$ Since $R_0^M$ is greater than $1$, this will lead to an outbreak of the disease and a spike in the number of cases. We perform several simulations for different values of $C$. Throughout the simulations we will fix the parameters of the systems to the values in Table~\ref{tab:parameters}, which correspond to the particular case of dengue.

\begin{table}[h]
	\centering
	\caption{Parameter values for dengue}
	\begin{adjustbox}{max width=\textwidth}
		\begin{tabular}{|c|c|c|c|c|}
			\hline
			Category & Parameter & Name & Value & Source \\
			\hline 
			\multirow{17}{*}{Biology}&  $b_M$ & Wild mosquitoes birth rate  & 4.4 day$^{-1}$ & \cite{wolbachia,bossin}\tablefootnote{In the model studied in \cite{bossin} the term accounting for the birth of mosquitoes is not straightforwardly equivalent to the one in our model. Its value has been adapted in order to account for this difference.}\\
			\cline{2-5}
			& $b_W$ & {\it Wolbachia} infected birth rate & 3.96 day$^{-1}$ & \cite{wolbachia,bossin} \\
			\cline{2-5}
			& $d_M$ & Wild mosquitoes death rate  & 0.04 day$^{-1}$ & \cite{wolbachia,bossin} \\ 
			\cline{2-5}
			& $d_W$ & {\it Wolbachia} infected death rate &  0.044 day$^{-1}$  & \cite{wolbachia,bossin}  \\
			\cline{2-5}
			& $d_S$ & Sterile mosquitoes death rate  & 0.12 day$^{-1}$ & \cite{cont_ster_wolb} \\ 
			\cline{2-5}
			& $s_h$ & Cytoplasmic incompatibility level & 0.9 & \cite{wolbachia}  \\
			\cline{2-5}
			& $s_c$ & Competitiveness level & 0.9 &  \\
			\cline{2-5}
			& $K$ & Carrying capacity & 65234 \tablefootnote{$K$ to $H$ ratios present a huge variability in the literature. Indeed, the ratio may depend on numerous factors and may not be constant in time. Lacking solid evidence to pick a value, we choose $K$ such that the size of the mosquito population at equilibrium is equal to the human population size.}   &  \\
			\cline{2-5}
			& $b_H$ & Human birth/death rate & 0.013 year$^{-1}$ & \\
			\cline{2-5}
			& $\sigma_H$ & Human recover time & 0.2 day$^{-1}$ & \cite{ndii_hickson} \\
			\cline{2-5}
			& $H$ & Human population size & 65000 & \cite{Deployment_Yogyakarta_Dengue} \\
			\cline{2-5}
			& $\beta_{HM}$ & Transmission rate H$\rightarrow$M & 0.1647 day$^{-1}$& \cite{ndii_hickson}  \\
			\cline{2-5}
			& $\beta_{MH}$ & Transmission rate H$\leftarrow$M & 0.1647 day$^{-1}$& \cite{ndii_hickson}  \\
			\cline{2-5}
			& $\beta_{HW}$ & Transmission rate H$\rightarrow$W & 0.157 day$^{-1}$ & \cite{ndii_hickson}  \\
			\cline{2-5}
			& $\beta_{WH}$ & Transmission rate H$\leftarrow$W & 0.0785 day$^{-1}$ & \cite{ndii_hickson}  \\
			\cline{2-5}
			& $\gamma_M$ & Non infected incubation period & 0.186 day$^{-1}$ & \cite{EIP}  \\
			\cline{2-5}
			& $\gamma_W$ & {\it Wolbachia} infected incubation period & 0.146 day$^{-1}$ & \cite{EIP} \\
			\cline{2-5}
			& $\gamma_H$ & Human incubation period & 0.17 day$^{-1}$ & \cite{chan}  \\
			\hline
			\multirow{2}{*}{Optimization} &  $T$ & Final time & 450 days &  \\
			\cline{2-5}
			&  $C$ & Amount of mosquitoes released  & $10^4$ - $6\cdot10^7$  & \\
			\hline
		\end{tabular}
		\label{tab:parameters}
	\end{adjustbox}
\end{table}

\subsection{Sterile Insect Technique}
The optimal solution for problem \eqref{prob:IH} in the SIT setting consists of a combination of consecutive pulses with a similar spacing. The fact that several spaced jumps are more efficient in reducing the number of susceptible mosquitoes, eventually leading to a reduction in the number of infections, is a result of the fact that, in the present model, the amount of sterile mosquitoes decreases exponentially between releases. Therefore, by spacing the releases a population of sterile mosquitoes can be sustained longer than doing one single release with all the mosquitoes together. We also observe that results do not only depend on the amount of mosquitoes released, but also on the number of releases considered. Comparing Figures~\ref{fig:Pulse_SIT_10} and \ref{fig:Pulse_SIT_20}  we can see how, by increasing the number of releases from 10 to 20, the final amount of infections is considerably reduced, specially with a comparatively high amount of mosquitoes. Nevertheless, this trend does not continue indefinitely. Increasing the number of releases way above 20 does not reduce significantly the number of infections anymore, even though there is no clear threshold in the number of releases for which the reduction is significant and that this may depend on the $C$ considered. The times and costs of the instant releases in Figures~\ref{fig:Pulse_SIT_10} and \ref{fig:Pulse_SIT_20} are given in Tables~\ref{tab:SIT_10} and \ref{tab:SIT_20} respectively.

\begin{table}[h!]
	\centering
	\caption{Results of the simulations performed for the SIT with 10 releases}
	\begin{tabular}{|c|c|c|c|c|}
		\hline
		$C$ & Time of releases &  Amount of mosquitoes released & $\int_0^T I_H(t)dt $ \\
		\hline 
		\multirow{5}{*}{$3 \cdot 10^7$} & $t_1=172.0$, $t_2= 178.4$ & $c_1=2277164.9$, $c_2=3118801.0$  & \multirow{5}{*}{$250375.4$} \\
		& $t_3=185.6$, $t_4=193.0$  & $c_3=3457741.0$, $c_4=3525953.9$ & \\
		& $t_5=200.7$, $t_6=208.7$  & $c_5=3458904.6$, $c_6=3328601.2$ & \\
		& $t_7= 217.3$, $t_8=226.6$  & $c_7=3157284.8$, $c_8=2932013.1$ & \\
		& $t_9=237.0$, $t_{10}=249.1$  & $c_9=2615241.0$, $c_{10}=2128294.3$ & \\
		\hline
		\multirow{5}{*}{$6 \cdot 10^7$} & $t_1=78.8$, $t_2=90.3$ & $c_1=6568442.3$, $c_2=8417318.4$  & \multirow{5}{*}{$72862.0$} \\
		& $t_3=102.9$, $t_4=116.3$  & $c_3=8619401.9$, $c_4=8082975.5$ & \\
		& $t_5=130.6$, $t_6=146.3$  & $c_5=7239149.0$, $c_6=6225676.6$ & \\
		& $t_7=163.5$, $t_8= 182.9$  & $c_7=5173640.8$, $c_8=4146284.4$ & \\
		& $t_9=205.0$, $t_{10}=230.8$  & $c_9=3194370.6$, $c_{10}=2332740.8$ & \\
		\hline
	\end{tabular}
	\label{tab:SIT_10}
\end{table}

\begin{table}[h!]
	\centering
	\caption{Results of the simulations performed for the SIT with 20 releases}
	\begin{tabular}{|c|c|c|c|c|}
		\hline
		$C$ & Time of releases &  Amount of mosquitoes released & $\int_0^T I_H(t)dt $ \\
		\hline 
		\multirow{10}{*}{$3 \cdot 10^7$} & $t_1=167.7$, $t_2=171.5$ & $c_1=1084557.5$, $c_2=1529725.4$  & \multirow{10}{*}{$244012.2$} \\
		& $t_3=175.7$, $t_4=179.9$  & $c_3= 1720612.9$, $c_4=1786314.4 $ & \\
		& $t_5= 184.0$, $t_6=188.1$  & $c_5=1793907.7$, $c_6=1781311.8$ & \\
		& $t_7=192.1$, $t_8= 196.1$  & $c_7=1750939.8$, $c_8= 1708089.2$ & \\
		& $t_9= 200.2$, $t_{10}=204.4$  & $c_9=1674451.1$, $c_{10}=1653637.3$ & \\
		& $t_{11}=208.7$, $t_{12}=213.2$  & $c_{11}=1641097.9$, $c_{12}=1597815.1$ & \\
		& $t_{13}=217.8$, $t_{14}=222.7$  & $c_{13}=1544202.8$, $c_{14}=1491664.7$ & \\
		& $t_{15}=227.8$, $t_{16}= 233.3$  & $c_{15}=1438846.7$, $c_{16}=1380652.5$ & \\
		& $t_{17}=239.1$, $t_{18}=245.5$  & $c_{17}=1308476.4$, $c_{18}=1199302.0$ & \\
		& $t_{19}=252.4$, $t_{20}=259.9$  & $c_{19}=1056512.9$, $c_{20}=857882.1$ & \\
		\hline 
		\multirow{10}{*}{$6\cdot 10^7$} & $t_1=0.0$, $t_2=3.7$ & $c_1=4230525.4$, $c_2=4214863.5$  & \multirow{10}{*}{$2124.4$} \\
		& $t_3=8.2$, $t_4=13.0$  & $c_3=4175080.6$, $c_4=4104782.5$ & \\
		& $t_5=18.1$, $t_6= 23.4$  & $c_5=4025147.5$, $c_6=3942640.9$ & \\
		& $t_7=29.2$, $t_8=35.5$  & $c_7=3855009.1$, $c_8=3759644.0$ & \\
		& $t_9=42.4$, $t_{10}=50.2$  & $c_9=3651466.3$, $c_{10}=3522392.6$ & \\
		& $t_{11}= 59.0$, $t_{12}=69.1$  & $c_{11}=3362768.5$, $c_{12}=3162408.6$ & \\
		& $t_{13}=80.8$, $t_{14}=94.2$  & $c_{13}=2913468.2$, $c_{14}=2614816.3$ & \\
		& $t_{15}=109.9$, $t_{16}= 128.1$  & $c_{15}=2275534.3$, $c_{16}=1912277.7$ & \\
		& $t_{17}=149.5$, $t_{18}=174.9$  & $c_{17}=1548925.6$, $c_{18}=1209010.4$ & \\
		& $t_{19}=205.7$, $t_{20}=243.9$  & $c_{19}= 911714.1$, $c_{20}=607524.1$ & \\
		\hline
	\end{tabular}
	\label{tab:SIT_20}
\end{table}

As we can observe in Figures~\ref{fig:Pulse_SIT_10} and \ref{fig:Pulse_SIT_20}, with a comparatively low amount of mosquitoes, $C=3\cdot 10^7$, the releases are concentrated around the peak of the infections, with the largest releases occurring  during the peak. Their effect is of only mitigating the outbreak, that is, the curve of infections remains fairly similar but peaking a bit earlier and lower. For this amount of mosquitoes we do not observe a great reduction in the number of cases by using 20 releases instead of 10. Namely, for 10 releases we obtain $J_{10}= 250375.4$ and for 20, $J_{20}=244012.2$. We can compute the reduction in the number of cases by comparing, numerically, $J(u)$ for the uncontrolled system with $J(u)$ for the controlled one. The value of $J(u)=\int_0^TI_H(t)dt$ in the case of the uncontrolled system is $J_{0}=293644.1$. This means that with $C=3\cdot 10^7$ we obtain approximately a $14.7\%$ reduction in the total amount of cases for 10 releases and a $16.9\%$ reduction for 20.

A possible interpretation for this solution is that, with the amount of mosquitoes considered, the population of susceptible mosquitoes cannot be consistently kept low for a long period. Therefore, the best use of the sterile males is to release them to reduce as much as possible the amount of susceptible (and also of infectious) mosquitoes when the transmission is at its prime.

On the other hand, with a comparatively big amount of mosquitoes, $C=6\cdot 10^7$, the releases shift to the beginning and present an asymmetrical, skewed shape. We see this happening with 10 releases, attenuating considerably further the outbreak, but even more in the case  with 20 releases. In this case, the first release occurs at $t_1=0.0$ and it results in an almost complete eradication of the outbreak. The largest releases occur soon after the first one. Releases get more sparse and smaller as time advances, specially for 20 releases, where some of them are clearly detached from the rest and occur after the peak of the outbreak. The fact that mosquitoes keep being released once the outbreak is suppresed is related to the fact that our model does not incorporate an Allee effect. This means that even when the wild mosquito population is very low, it can grow again to its initial values if the releases of sterile mosquitoes stop. Therefore, releases of small amounts of mosquitoes are needed so the outbreak does not start again inside the time horizon considered. The difference observed as a result of the different number of jumps in this case is more abrupt. We obtain a value of $J_{10}=72862.0$ for 10 releases, which means  $75.2\%$ less infections in the time window considered, and a value of $J_{20}=2124.4$ for 20 releases, that is, a $99.3\%$ reduction.  

With this amount of mosquitoes, specially when they are spread over 20 releases, the population can be kept low for a long time. Our interpretation of these results is that, due to this capability of long term population reduction, the optimal solution consists in releasing as soon as possible, preventing the outbreak from gaining traction in the first place. Then, smaller releases keep being done to prevent the population from increasing again. Hence, being able to divide the mosquitoes in more releases becomes more important in this case. We see clearly  how the number of releases can affect the outcome, even for the same $C$, in the lower rows of figures \ref{fig:Pulse_SIT_10} and \ref{fig:Pulse_SIT_20}. With 10 releases, although initially the outbreak is greatly reduced, the population cannot be kept low consistently in all the time window considered and the cases rise again substantially towards the end.

\begin{figure}[h!]
	\centering
	\includegraphics[width=0.95\textwidth]{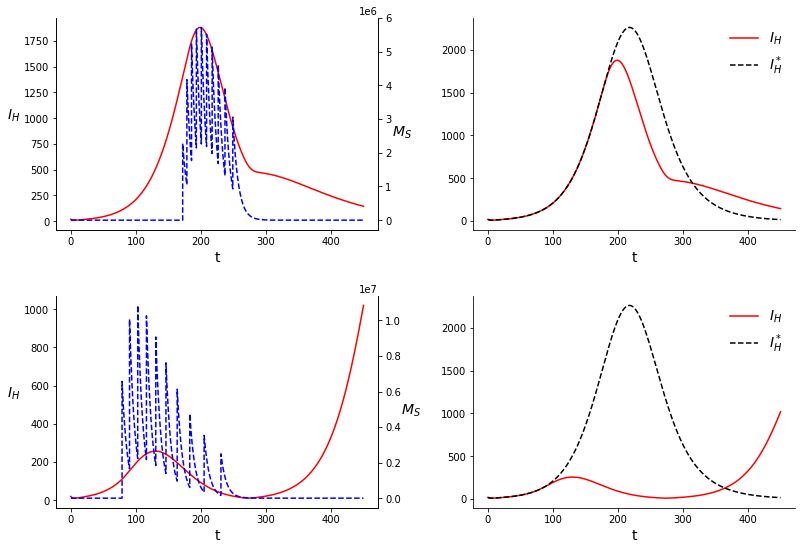}
	\caption{Results of the simulations for the SIT with $C=3\cdot 10^7$ (upper row) and $C=6\cdot 10^7$ (lower row) considering 10 releases. The dashed blue line corresponds to the amount of sterile mosquitoes present at each time and the its jumps correspond to the releases. $I_H^*$, on the right column, corresponds to the uncontrolled case.}
	\label{fig:Pulse_SIT_10}
\end{figure}

\begin{figure}[h!]
	\centering
	\includegraphics[width=0.95\textwidth]{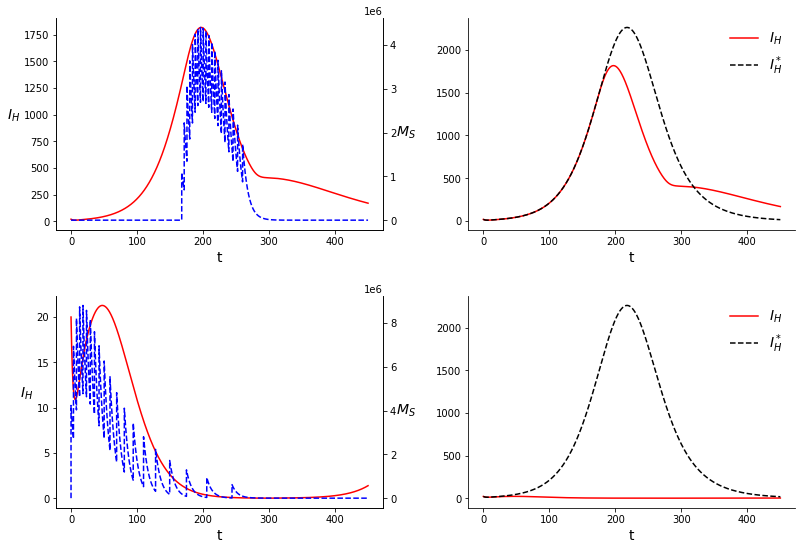}
	\caption{Results of the simulations for the SIT with $C=3\cdot 10^7$ (upper row) and $C=6\cdot 10^7$ (lower row) considering 20 releases.}
	\label{fig:Pulse_SIT_20}
\end{figure}

Another approach we can take, arguably more in line with applications on the field, is to optimize only the times of the releases while keeping the amount of mosquitoes constant which may correspond, for instance, to logistic constraints like a constant production capacity of the sterile mosquito production facility and/or a limited flow capacity for the mosquito release protocol. Of course, the result in the reduction of the infections will be worse than the counterpart we have just presented. Nevertheless, it raises a reasonable question that we can answer thanks to our model simulations: to which extent it is preferable the use of a more sophisticated technique over a less efficient but simpler one? The results are presented tables \ref{tab:SIT_10_Cfix} and \ref{tab:SIT_20_Cfix}.

\begin{table}[h!]
	\centering
	\caption{Results of the simulations performed for the SIT with 10 releases and all $c_i=C/10$.}
	\begin{tabular}{|c|c|c|c|}
		\hline
		$C$ & Time of releases & $\int_0^T I_H(t)dt $ \\
		\hline 
		\multirow{2}{*}{$3 \cdot 10^7$} & $t_1=173.2$, $t_2=180.6$, $t_3= 187.4$, $t_4=194.1$, $t_5=201.1$ & \multirow{2}{*}{$250880.3$} \\
		& $t_6=208.3$, $t_7=216.3$, $t_8=225.1$, $t_9=235.4$, $t_{10}=248.0$ &  \\
		\hline
		\multirow{2}{*}{$6\cdot 10^7$} & $t_1=98.4$, $t_2=109.2$, $t_3=119.5$, $t_4=130.1$, $t_5=141.5$ & \multirow{2}{*}{$99223.3$} \\
		& $t_6=154.3$, $t_7=169.0$, $t_8=186.5$, $t_9=208.3$, $t_{10}=236.4$  & \\
		\hline
	\end{tabular}
	\label{tab:SIT_10_Cfix}
\end{table}

\begin{table}[h!]
	\centering
	\caption{Results of the simulations performed for the SIT with 20 releases and all $c_i=C/20$.}
	\begin{adjustbox}{max width=\textwidth}
		\begin{tabular}{|c|c|c|c|}
			\hline
			$C$ & Time of releases & $\int_0^T I_H(t)dt $ \\
			\hline 
			\multirow{4}{*}{$3 \cdot 10^7$} & $t_1=168.3$, $t_2=172.8$, $t_3=176.8$, $t_4=180.6$, $t_5=184.2$ & \multirow{4}{*}{$244623.4$} \\
			& $t_6=187.8$, $t_7=191.4$, $t_8=195.0$, $t_9=198.7$, $t_{10}=202.5$ &  \\
			& $t_{11}=206.5$, $t_{12}=210.6$, $t_{13}=214.9$, $t_{14}=219.5$, $t_{15}=224.4$ &  \\
			& $t_{16}=229.8$, $t_{17}=235.7$, $t_{18}=242.3$, $t_{19}=250.0$, $t_{20}=259.5$ &  \\
			\hline
			\multirow{4}{*}{$6\cdot 10^7$} & $t_1=0.0$, $t_2=3.8$, $t_3=8.0$, $t_4=12.4$, $t_5=17.0$ & \multirow{4}{*}{$2556.1$} \\
			& $t_6=21.7$, $t_7=26.8$, $t_8=32.1$, $t_9=38.0$, $t_{10}=44.4$ &  \\
			& $t_{11}=51.6$, $t_{12}=59.6$, $t_{13}=68.9$, $t_{14}=79.7$, $t_{15}=92.6$ &  \\
			& $t_{16}=108.1$, $t_{17}=127.4$, $t_{18}=151.7$, $t_{19}=183.3$, $t_{20}=225.5$ &  \\
			\hline
		\end{tabular}
	\end{adjustbox}
	\label{tab:SIT_20_Cfix}
\end{table}

In figures \ref{fig:Pulse_SIT_10_Cfix} and \ref{fig:Pulse_SIT_20_Cfix} we can see that optimal strategies in time do not differ a lot with those of figures \ref{fig:Pulse_SIT_10} and \ref{fig:Pulse_SIT_20} respectively. Still, releases are done around the peak of the outbreak in the case of a relatively low amount of mosquitoes. As we increase the amount of mosquitoes and the number of releases they shift to the left,  resulting in a further reduction of the infections.

\begin{figure}[h!]
	\centering
	\includegraphics[width=0.95\textwidth]{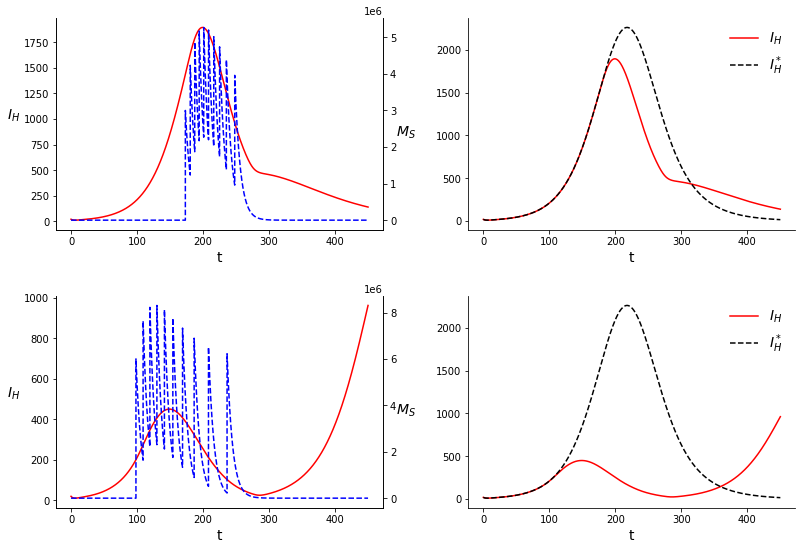}
	\caption{Results of the simulations for the SIT with $C=3\cdot 10^7$ (upper row) and $C=6\cdot 10^7$ (lower row) considering 10 releases and an equal distribution of the mosquitoes between the releases.}
	\label{fig:Pulse_SIT_10_Cfix}
\end{figure}

\begin{figure}[h!]
	\centering
	\includegraphics[width=0.95\textwidth]{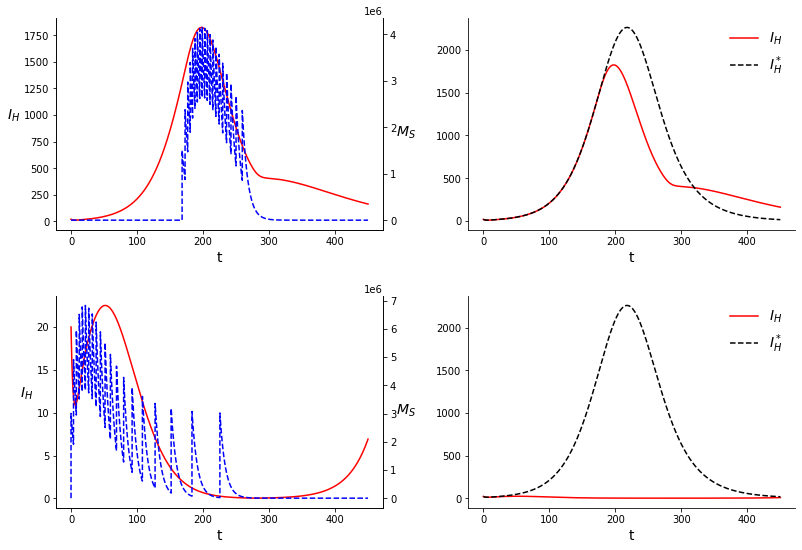}
	\caption{Results of the simulations for the SIT with $C=3\cdot 10^7$ (upper row) and $C=6\cdot 10^7$ (lower row) considering 20 releases and an equal distribution of the mosquitoes between the releases.}
	\label{fig:Pulse_SIT_20_Cfix}
\end{figure}

As for the effectiveness of this approach, we show a comparison of the results of the optimization of the times alone and that of the times and the costs in tables \ref{tab:SIT_10_Comp} and \ref{tab:SIT_20_Comp}. As we can see there does not seem to be a significant advantage in optimizing both times and costs except in one case, the case with $C=6\cdot 10^7$ and 10 releases. 

The fact that the optimal strategy can change significantly when the number of releases is increased suggests that solutions for this setting are very sensitive to changes on the problem characteristics. Comparing figures \ref{fig:Pulse_SIT_10} with \ref{fig:Pulse_SIT_10_Cfix}, and \ref{fig:Pulse_SIT_20} with \ref{fig:Pulse_SIT_20_Cfix}, we see that optimizing the amount of mosquitoes at each release makes the first releases move to the left but also increases the total time span of the releases, which is a similar effect to the addition of new releases. Elaborating further in our biological interpretation of the results, this suggests that for ten releases and $C=6\cdot 10^7$ we can keep the population low during a certain amount of time, but not enough to prevent the outbreak. A slight improvement of the technique in this setting (either an increase in the number of releases or an optimization of the number of mosquitoes released at each impulse) can make a difference in the ability to control the outbreak by keeping the wild mosquito population at a low level over a longer period of time, thus improving the results. On the other hand, when we are far from significantly reducing the outbreak or when we can almost prevent it, the advantage of also optimizing the amount of mosquitoes at each release becomes smaller.

\begin{table}[h!]
	\centering
	\caption{Comparison of the reductions in the infections obtained on the simulations performed for the SIT with 10 releases.}
	\begin{tabular}{|c|c|c|}
		\hline
		$C$ & Times & Times and costs \\
		\hline 
		$3 \cdot 10^7$ & $14.6\%$ & $14.7\%$ \\
		\hline
		$6\cdot 10^7$ & $66.2\%$ 
		& $75.1\%$ \\
		\hline
	\end{tabular}
	\label{tab:SIT_10_Comp}
\end{table}

\begin{table}[h!]
	\centering
	\caption{Comparison of the reductions in the infections obtained on the simulations performed for the SIT with 20 releases.}
	\begin{tabular}{|c|c|c|}
		\hline
		$C$ & Times & Times and costs \\
		\hline 
		$3 \cdot 10^7$ & $16.7\%$ & $16.9\%$ \\
		\hline
		$6\cdot 10^7$ & $99.1\%$ 
		& $99.3\%$ \\
		\hline
	\end{tabular}
	\label{tab:SIT_20_Comp}
\end{table}

\subsection{{\it Wolbachia} method}

Regarding the {\it Wolbachia} method in all cases all the pulses cluster into one single pulse. In other words, the optimal solution is to perform a single release with all the available mosquitoes. This makes it useless to optimize the amount of mosquitoes released in each jump and turns the problem into a one-dimensional optimization one: $\min\limits_{t_1\in[0,T]} J(u)$, with $J(u)=\int_0^T I_H(t)dt$ and $u(t)=C\delta(t-t_1)$.

As found in other works studying the use of {\it Wolbachia} to produce a mosquito population replacement \cite{ABP,wolbachia}, solutions present two clearly distinct behaviours. Since the equation $p'=f(p)$ is bi-stable, if the proportion of {\it Wolbachia} infected mosquitoes exceeds a certain threshold, $p=\theta$, then the system moves to a full invasion state without further intervention. The parameter determining the two regimes is the total amount of mosquitoes, $C$. If there are more mosquitoes than the amount needed to lead the system to $p=\theta$ we will observe one kind of behaviour, different from the case where there are less. From the initial conditions we have $p(0)=0$. We can compute the amount of mosquitoes needed to reach $p=\theta$ in a single jump. If we reach $p=\theta$ in the first jump, $\theta=p(t_1^+)=G^{-1}\lrp{G(p(t_1^-)+C)}=G^{-1}(C)$, thus $C=G(\theta)$. For the parameters considered here, $G(\theta)\approx 14850$. 

In figure \ref{fig:Pulse_WB} we plot the optimal solutions to problem \eqref{prob:IH} for system \eqref{sys:Wolbachia_Simp} with the parameters of table \ref{tab:parameters}. In case $C<G(\theta)$ the jump occurs before the outbreak reaches its peak. The larger is $C$, the smaller is $t_1$. In Figure \ref{fig:Pulse_WB}, for $C=10000$, $t_1 = 147.5$. Instead, in case $C>G(\theta)$ the jump is at $t_1=0$. The system from this point tends to $p=1$ without the need of releasing mosquitoes anymore. The value of $J(u)=\int_0^TI_H(t)dt$ in the case of the uncontrolled system yields $J_{0}=294501.4$. With $C=10000$ the profile of the outbreak is not altered very much, but it peaks at a lower value. The value of $J(u)$ in this case is $J_{10000}=288362.7$, roughly a $2.1\%$ reduction in the total amount of cases.
With $C=20000$ the change is the infected humans curve is much more appreciable. The curve peaks at a much lower level but decays slower. In this case the value of $\int_0^TI_H(t)dt$ is $J_{20000}=128899.1$, which is a $56.2\%$ reduction in the number of cases.

\begin{figure}[h!]
	\centering
	\includegraphics[width=0.95\textwidth]{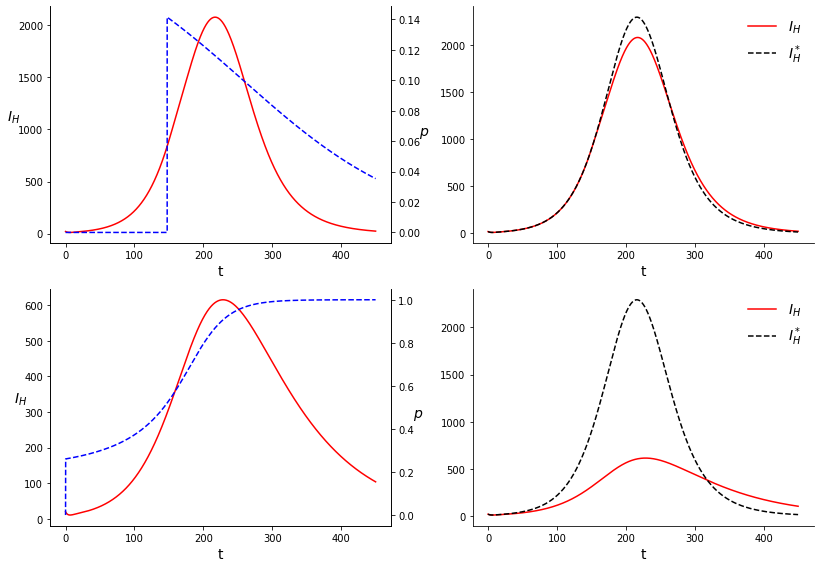}
	\caption{Results of the simulations for the {\it Wolbachia} method with $C=10000$ (upper row) and $C=20000$ (lower row). The proportion of {\it Wolbachia} infected mosquitoes corresponds to the dashed blue line on the left column. $I_H^*$, on the right column, corresponds to the uncontrolled case.}
	\label{fig:Pulse_WB}
\end{figure}

The biological interpretation of these results is in line with the one for the sterile mosquitoes. When it is not possible to trigger a population replacement, the optimal strategy is to release the mosquitoes before the peak of the epidemic. Since the number of Wolbachia-infected mosquitoes declines with time, this policy minimizes the presence of the wild mosquitoes (with a greater vector capacity) during the phase of largest transmission. On the other hand, if it is possible to trigger the population replacement, the sooner we act in the system, the better. Since the proportion of Wolbachia-infected mosquitoes is going to increase naturally there are no incentives for waiting to make the release. 

We remark that the amount of mosquitoes needed for this technique to be effective is much lower than for the SIT. This makes sense, since the {\it Wolbachia} population is self-sustainable while the sterile mosquitoes do not reproduce and thus it is necessary to continue to do new releases if we don't want the sterile mosquito population to die out. Nonetheless, the exact values of mosquitoes released, or the ratio of mosquitoes needed in one technique with respect to the other cannot be drawn directly from our study due to the limitations of the model and the uncertainty on the parameters.

\section{Numerics: an augmented Lagrangian algorithm}\label{sec:Uzawa}

In this section, we explain and detail further the numerical method used for obtaining the results. We implemented a gradient descent to optimize the times of the releases, $t_i$. At each step, the coefficients $(c_i)_{1\leq i\leq n}$ being given, the control function was updated according to $$u_{k+1}=\Pi_{\mathcal{T}} \left(u_k-\varepsilon_t\nabla_{t}J(u_k)\right) \mbox{, where} \quad \nabla_t J(u)= \lrp{\delta_{t_1} J(u),\dots,\delta_{t_n} J(u)}$$
and where $\Pi_{\mathcal{T}}$ denotes the projection onto the set of controls 
$$
\left\{\sum_{i=1}^{n}c_i\delta(t-t_i), \ 0\leq t_1\leq \dots\leq t_n\right\}.
$$

Here, $J(u)=\int_0^T I_H(t)dt$. The values of $\delta_{t_i} J(u)$ for $i=1,\dots,n$ have been computed in Proposition~\ref{prop:deltatk} (see Section \ref{sec:optcond}).

Starting from a random initial condition we optimize the time of the releases, $t_i$, until a certain level of functional flatness is attained. Then we optimize the $c_i$, that is, the amount of mosquitoes released at each $t_i$.

The costs, $c_i$, have been optimized using an augmented Lagrangian algorithm, which consists in considering the following functional 
$$
L(u,\lambda)=\int_0^TI_H(t)dt+\lambda\left(\sum_{i=1}^{n}c_i-C\right)+\frac{\rho}{2}\left(\sum_{i=1}^{n}c_i-C\right)^2.
$$
The second term is added in order to take into account the constraint $\sum_{i=1}^{n}c_i=C$. The real number $\lambda$ is the Lagrange multiplier associated this constraint, which has to be found numerically at the same time as $u$. The augmented Lagrangian method transforms the constrained minimization problem into an unconstrained one, similarly to the Uzawa algorithm. The new functional has to be minimized with respect to $u$, and maximized with respect to $\lambda$. The solution to the problem is hence searched as a saddle point of $L$. The addition of the third term can be seen as a convexification of the dual problem. The addition of the squared term  to the Lagrangian accelerates the convergence provided $\rho$ is chosen carefully.

In order to find the saddle point of $L$ we take one step at a time, minimizing with respect to $u$ and then maximizing it with respect to $\lambda$, following the scheme:
\begin{equation*}
	\begin{array}{lll}
		u_{k+1} & = & u_k-\varepsilon_c\left(\nabla_{c}J(u_k)+\lambda_{k}+\rho\left(\sum_{i=1}^{n}c_i-C\right)\right), \\
		\lambda_{k+1} & = & \max\left(\lambda_k+\rho\left(\sum_{i=1}^{n}c_i-C\right),0\right). \\
	\end{array}
\end{equation*}
Where $\nabla_c J(u)$ is the gradient of the functional $J(u)$ with respect to the costs, analogous to $\nabla_t J(u)$. The components of $\nabla_c J(u)$ have been computed in Proposition \ref{prop:deltack} (see Section \ref{sec:optcond}). Additional explanations regarding \textsl{augmented Lagrangian type algorithms} can be found in \cite{MR2441683}.

In order to picture better the algorithm implemented we provide in figure \ref{fig:J_constr} an example of history of two key quantities along the iterations of the algorithm, namely $J(u)=\int_0^T I_H(t)dt$ and $\sum_{i}c_i-C$ during a simulation. We take as an example the simulation for the sterile insect technique with 20 releases and $C=3\cdot10^7$. The value of $J$ falls sharply at the begginning as the times of the releases, $(t_i)_{1\leq i\leq n}$, move from their initial random positions. The small oscillations observed later correspond to the first time we optimize the weights, $(c_i)_{1\leq i\leq n}$. Since we are looking for a saddle point of $L$ there are iterations where the value of $J$ actually increases. Then it starts a slow convergence to the final state where the values of $(t_i)_{1\leq i\leq n}$ and $(c_i)_{1\leq i\leq n}$ are refined. The simulation stops when a certain level of functional flatness is attained. In Figure~\ref{fig:J_constr}, on the right, the x-axis presents slightly less iterations since we only show the iterations on the weights. At first this quantity oscillates until the value of $L$ stabilizes. Since we alternate the optimization of the times and the weights, whenever the times are adjusted, new oscillations appear as the weights, $(c_i)_{1\leq i\leq n}$, adjust to the new $(t_i)_{1\leq i\leq n}$ values. As expected, in the long run, $\sum_i c_i -C$ stabilizes around 0, so the constraint $\sum_i c_i=C$ is respected.

\begin{figure}[h!]
	\centering
	\includegraphics[width=\textwidth]{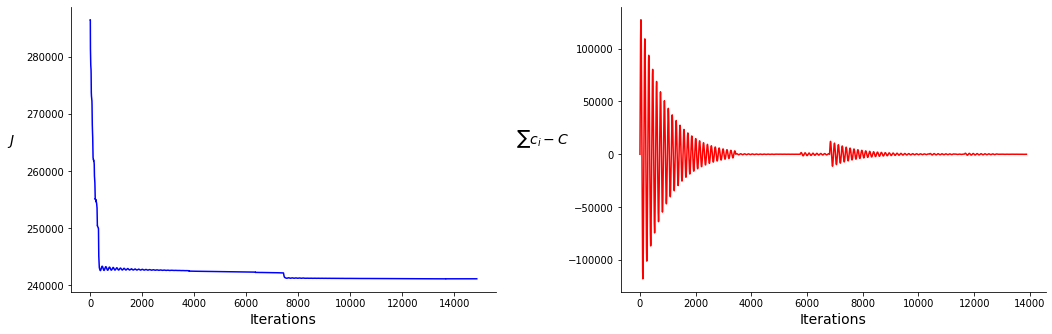}
	\caption{Evolution of the functional $J(u)$ and $\sum c_i -C$ during the sterile insect simulation for 20 releases and $C=3\cdot 10^7$.\label{fig:J_constr}}
\end{figure}

\section{Conclusion and perspectives}

In this paper, we have focused on presenting an approach combining mathematical modeling and optimization, which we believe is relevant for determining good mosquito release strategies. We focused on simplified models that are nevertheless relevant in certain regimes and parameter ranges. Despite the simplifications, they allow us to single out different strategies depending on some simple logistic constraints like the number of mosquitoes available to be released, in total or each time, or the number of releases that can be done. SIT and population replacement using \textit{Wolbachia} have different goals and are very different strategies when it comes to long term planning. Our results single out these differences, but they also find some similarities when these strategies are applied in the context of epidemic outbreak control. They can be summed up as follows:
\begin{itemize}
	\item In both cases, when we have at our disposal fewer mosquitoes (or also in case we do not consider enough
	releases for the SIT) the optimal strategy is focused on the \textit{mitigation} of the outbreak. When resources do not allow for a proper control of the outbreak, independently of the particular control mechanism considered, they should be allocated when most of the infections are occurring in order to dampen the transmission, thus mosquitoes are released before or around the peak of the epidemic.
	
	\item In case we have enough mosquitoes (and we consider enough releases in the SIT case), the optimal strategy shifts to the \textit{suppression} of the outbreak. When the control of the outbreak is possible, action must be taken from the beginning of the time window considered, preventing the epidemic from gaining traction in the first place, which allows to reduce drastically the number of infections in both scenarios.
	
	\item The main differences observed in the context considered are, first, the need for multiple releases in SIT (since sterile mosquitoes do not reproduce and need to be periodically replaced), while for \textit{Wolbachia} population replacement, under our hypothesis, the optimal release strategy consists in one single release. Another important difference is the sharp transition between both regimes observed with the \textit{Wolbachia} technique, whereas in the SIT case this transition between regimes is smooth.
	
\end{itemize}

Although, as previously mentioned, we don't claim that the present study allows to do a precise comparison of the advantages of the different strategies, the results of this study suggest that it should be worth applying the same methodology  to more complex and general models. In particular, we would like to focus in the future on the following generalizations:

\begin{itemize}
\item As explained in Remark \ref{rem:males}, only female mosquito dynamics is considered in our models, while male dynamics is assumed to be similar. We wish to consider more precise models, in which the specific dynamics of males is also taken into account. This will require the addition of extra compartments in the model and will increase the dimension of the differential system considered.

\item In the case of the SIT, results vary significantly with the number of releases considered when this one is low, but the improvements dampen as the number of releases increase. It can be useful to properly study the improvement of results as a function of the number of releases considered for different values of C, in order to be able to estimate in which cases it is worth to consider a bigger number of releases.

\item Also in the SIT case, results depend on the time window considered, since mosquitoes can reproduce again in treated areas when the treatment is stopped. A very interesting question would be: How does the optimal strategy evolve as T increases? In other words, how do optimal strategies evolve when we do not restrict ourselves to the duration of a particular outbreak but instead we want to minimize the infections for, a priori, unbounded periods of time?

\item In the search for optimal strategies, another relevant factor we wish to take into account is seasonality, which has a strong influence on mosquito reproduction rates and development.
\end{itemize}

\vspace{0.2cm}
Data (the code that was used for doing the simulations) is available on reasonable request.
\vspace{0.2cm}

\section*{Acknowledgement}
	
\begin{minipage}{1.9cm}
\includegraphics[width=\textwidth]{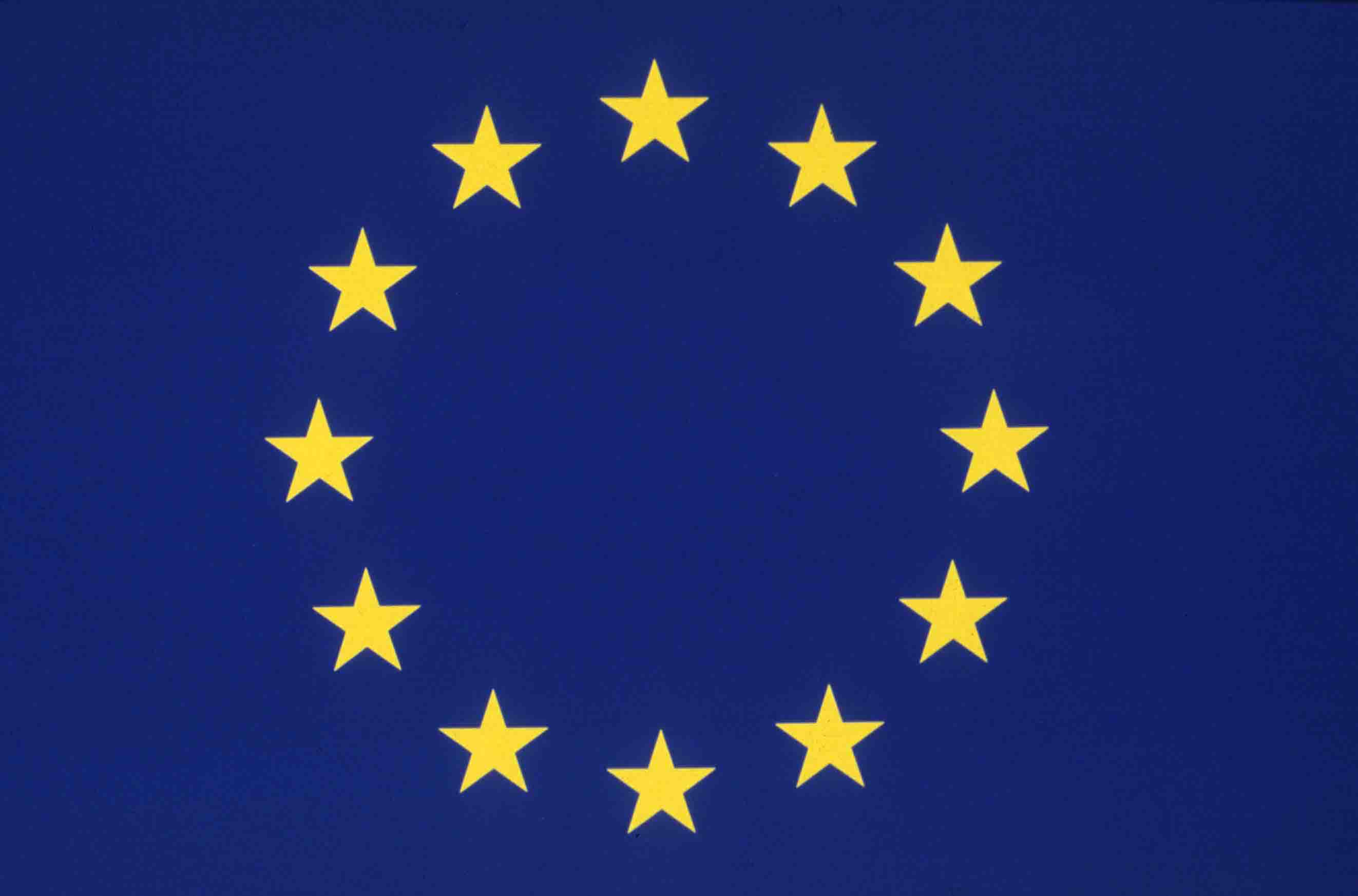}
\end{minipage}
\hspace{0.2cm}
\begin{minipage}{13.7cm}
This project has received funding from the European Union’s Horizon 2020 research and innovation programme under the Marie Skłodowska-Curie grant agreement No 754362.
\end{minipage}

\vspace{0.2cm}

It also benefited from the support of the PHC/FCT Program Pessoa 2020 (44567YH-5460).
C.~Rebelo has been supported by FCT projects  UIDB/04621/2020 and UIDP/04621/2020 of CEMAT at FC-Universidade de Lisboa.

The first three authors were partially supported by the Project ``Analysis and simulation of optimal shapes - application to lifesciences'' of the Paris City Hall.

The third author were partially supported by the ANR Projects Shape Optimization - ``SHAPO'' and ``TRECOS''.

\bibliographystyle{abbrv}%{plain}
\bibliography{bib_dengue}

\end{document}